\documentclass[12pt]{article}

\usepackage{amsmath}
\usepackage{amsfonts}
\usepackage{amsthm}
\usepackage{graphicx}
\usepackage{wrapfig}
\usepackage{biblatex}
\usepackage{hyperref}
\usepackage{tikz}
\usepackage{float}

\newtheorem{theorem}{Theorem}
\newtheorem{proposition}{Proposition}
\newtheorem{lemma}{Lemma}
\newtheorem{corollary}{Corollary}
\bibliography{bibliography}

\theoremstyle{definition}
\newtheorem{example}{Example}
\newcommand{\spq}{s_{p,q}}

\title{The Orbits of the Action of the Cactus Group on Arc Diagrams}
\author{Matvey Borodin}

\begin{document}
\maketitle

\begin{abstract}
    The cactus group $J_n$ is the $S_n$-equivariant fundamental group of the real locus of the Deligne-Mumford moduli space of stable rational curves with marked points. This group plays the role of the braid group for the monoidal category of Kashiwara crystals attached to a simple Lie algebra. Following Frenkel, Kirillov and Varchenko, one can identify the multiplicity set in a tensor product of $\mathfrak{sl}_2$-crystals with the set of \emph{arc diagrams} on a disc, thus allowing a much simpler description of the corresponding $J_n$-action. We address the problem of classifying the orbits of this cactus group action. Namely, we describe some invariants of this action and show that in some (fairly general) classes of examples there are no other invariants. Furthermore, we describe some additional relations, including the braid relation, that this action places on the generators of $J_n$.
    
    
\end{abstract}

\section{Introduction}\label{introduction}

\subsection{Motivation}
 Kashiwara crystals attached to a simple Lie algebra $\mathfrak{g}$ are combinatorial models of representations of $\mathfrak{g}$, where the weight spaces are represented by finite sets and the action of Chevalley generators are represented by arrows connecting their elements in such a way that the weights match. The decomposition of a representation into a direct sum of irreducibles corresponds to the decomposition of a crystal into connected components. Kashiwara crystals naturally arise as indexing sets for canonical bases in finite-dimensional representations of the corresponding quantum group $U_q(\mathfrak{g})$, see \cite{Lu}.
 
 There is a well-defined tensor product operation on crystals that captures the combinatorial (Littlewood-Richardson) rules of decomposing tensor products of irreducible representations. This tensor product is not symmetric and not even braided. Rather, it has a different property that gives rise to a natural action of the \emph{cactus group} $J_n$ on any $n$-fold tensor product, as described in \cite{henriques2005crystals}. As defined in~\cite{DAVIS2003115}, $J_n$ is the $S_n$-equivariant fundamental group of $\overline{M_{0,n+1}}(\mathbb{R})$, the real locus of the Deligne-Mumford moduli space of stable rational curves with marked points. The cactus group $J_n$ can be defined using generators and relations. Consider the set of generators $\{\spq \mid 1 \leq p < q \leq n\}$. We define the following relations on these generators:
 \begin{itemize}
     \item $\spq^2 = e$ where $e$ is the identity for any $\spq$.
     \item $\spq s_{p',q'} = s_{p',q'} \spq$ if $q' < p$ or $q < p'$, that is the intervals $[p,q]$ and $[p', q']$ are disjoint.
     \item $\spq s_{p',q'} \spq = s_{p+q-q', p+q-p'}$ if $p \leq p' < q' \leq q$, that is the interval $[p', q']$ falls inside the interval $[p,q]$.
 \end{itemize}

In \cite{FKV} the (dual) canonical basis in a tensor product of finite-dimensional irreducible $U_q(\mathfrak{sl}_2)$-modules $V_{l_1}^q\otimes\dots\otimes V_{l_n}^q$ is constructed in terms of the Schechtman-Varchenko realization of $U_q$-modules in the homology of an appropriate local system on the configuration space of colored points on the complex line \cite{SV}. 

The canonical basis in the space of highest vectors in the $l_\infty$-weight space of the tensor product $V_{l_1}^q\otimes\dots\otimes V_{l_n}^q$ are given by \emph{complete arc diagrams}. We define a complete arc diagram as follows: consider real projective line with $n$ finite points labeled $z_1, z_2, \dots, z_n$ and a point at infinity labeled $z_\infty$. We picture this line as the boundary of a disc and construct non-intersecting arcs connecting the points such that each point $z_i$ or $z_\infty$ serves as the endpoint exactly $l_i$ (respectively, $l_\infty$) arcs. We denote the set of complete arc diagrams corresponding to $\{l_1, l_2, \dots, l_n, l_\infty\}$ up to continuous deformations by $X(l_1, l_2, \dots, l_n, l_\infty)$.

We denote the irreducible crystal of highest weight $l$ as $\mathcal{B}_l$. Note that the space of highest vectors of weight $l_\infty$ in $V_{l_1}^q\otimes\dots\otimes V_{l_n}^q$ is the same as the Hom-space $\text{Hom}(V_{l_\infty}^q,V_{l_1}^q\otimes\dots\otimes V_{l_n}^q)$, hence the set of complete arc diagrams is naturally in bijection with $\text{Hom}(\mathcal{B}_{l_\infty},\mathcal{B}_{l_1}\otimes\dots\otimes \mathcal{B}_{l_n})$. One can define the action of the cactus group $J_n$ on $X(l_1, l_2, \dots, l_n, l_\infty)$ in elementary terms, so that we have bijections of $J_n$-sets, as shown in~\cite{Markarian_Rybnikov}:
$$
\text{Hom}(\mathcal{B}_{l_\infty},\mathcal{B}_{l_1}\otimes\dots\otimes \mathcal{B}_{l_n})\longleftrightarrow X(l_1, l_2, \dots, l_n, l_\infty).
$$

\subsection{Main problems}

There is a relatively simple description of the action of the cactus group $J_n$ on the set $X(l_1, l_2,\dots,l_n, l_\infty)$. We begin by assigning numbers between 1 and $n$ to the points $z_1, \dots, z_n$ starting from $z_\infty$ and going clockwise. Then $\spq$ acts by reversing the order of all points which were assigned numbers between $p$ and $q$, inclusive, and reconnecting the arcs so that the arcs remain non-intersecting. Note that the step where the arcs are reconnected is uniquely determined.

Any element $s_{p_i, q_i}\dots s_{p_2, q_2}s_{p_1, q_1} \in J_n$ acts by applying the actions corresponding to $s_{p_j, q_j}$ going in order from right to left for $1 \leq j \leq i$.

There are two natural questions one can ask about this action. Firstly, we consider how many orbits this action has on the set $X(l_1, l_2,\dots,l_n, l_\infty)$ and whether there are any universal invariants of this action. Furthermore, since the cactus group is infinite while the set of arc diagrams is finite, we consider what additional relations are imposed on the cactus group by this group action.

\subsection{The results}

We begin by describing a number of visual invariants of this action including border thickness in Lemma~\ref{lemma:border_thickness}, the greatest common divisor of the counts of connecting lines in Lemma~\ref{lemma:gcf} and the number of components in Lemma~\ref{lemma:components_invariant}.

In Theorem~\ref{thm:transitivity}, we prove that if at least one of the $l_i$ values (or $l_\infty$) is equal to 1, the action of the cactus group on $X(l_1, l_2, \dots, l_n, l_\infty)$ is transitive. 

In Theorem~\ref{thm:l_n=2}, we prove that in the case when all $l_i = 2$ (and $l_\infty = 2$), all orbits are classified by an invariant which we call the number of components. This allows us to explicitly describe all orbits of this action for any $n$. This leads to Corollary~\ref{corr:counting_orbits} in which we show that the number of orbits grows proportionally to $n$ as $n \to \infty$. 

In Theorem~\ref{thm:braid_relation}, we prove that this action of the cactus group imposes the braid relation, $s_{i, i+1} s_{i-1, i} s_{i, i+1} = s_{i-1, i} s_{i, i+1} s_{i-1, i}$, on the cactus group for any set $X(l_1, l_2, \dots, l_n, l_\infty)$. This result follows from~\cite{chmutov2017berensteinkirillov}, but we provide a new, simpler proof. In the special case when $l_1 = l_2 = \dots = l_\infty$, we also prove that the action imposes the relation ${(s_{1,n}s_{1,n-1})}^{n(n+1)} = e$ on the cactus group in Theorem~\ref{thm:second_relation}.

\subsection{Organization of the paper} In Section~\ref{sec:background} we begin with basic definitions and examples of the cactus group $J_n$. In Section~\ref{sec:action_definition} we define the set of arc diagrams and the action of $J_n$ on this set. In Section~\ref{sec:simple_invariants} we describe simple invariants of the action of $J_n$. In Section~\ref{sec:n=3_description} we fully describe the action of $J_n$ in the case when $n=3$. In Section~\ref{sec:transitivity1} we prove that when one of the $l_i = 1$ (or $l_\infty = 1$), the action of the cactus group is transitive. In Section~\ref{sec:single_invariant} we completely describe the orbits in the action of $J_n$ if all $l_i = 2$. Finally, in Section~\ref{sec:simple_relation} we prove a two relations beyond those that generate $J_n$ is always satisfied when $J_n$ acts on the set of arc diagrams.

\subsection{Acknowledgements}\label{sec:acknowledgements}

I would like to thank my mentor, Prof.~Leonid Rybnikov, for suggesting this topic and guiding me in my exploration of it. I would also like to thank the MIT PRIMES Program for making this project possible.

\section{The Cactus Group}\label{sec:background}

We define the cactus group $J_n$ using generators and relations. Consider the set of generators $\{\spq \mid 1 \leq p < q \leq n\}$. We define the following relations on these generators:
\begin{itemize}
    \item $\spq^2 = e$ where $e$ is the identity for any $\spq$.
    \item $\spq s_{p',q'} = s_{p',q'} \spq$ if $q' < p$ or $q < p'$, that is the intervals $[p,q]$ and $[p', q']$ are disjoint.
    \item $\spq s_{p',q'} \spq = s_{p+q-q', p+q-p'}$ if $p \leq p' < q' \leq q$, that is the interval $[p', q']$ falls inside the interval $[p,q]$.
\end{itemize}

The following proposition, which can be found in \cite{mostovoy2018pure}, highlights an important aspect of the structure of this group:

\begin{proposition}\label{prop:homomorphism_to_S}
    There exists a surjective homomorphism $\phi: J_n \to S_n$, where $S_n$ is the group of permutations on $n$ elements under composition. 
\end{proposition}

The kernel of this homomorphism is known as the \emph{pure cactus group}.

This gives us a simple example of a $J_n$ action.

\begin{example}
    The cactus group $J_n$ can act on the set of natural numbers $\{1, 2, \dots, n\}$. Consider $a \in J_n$ and $b \in \{1, 2, \dots, n\}$. We define $ab = \phi(a)b$ with $\phi$ as defined in Proposition~\ref{prop:homomorphism_to_S} and $\phi(a)b$ defined as the natural group action of $S_n$ on $\{1, 2, \dots, n\}$, that is the index to which $b$ is sent by the permutation. 
\end{example}

\section{The Action of \texorpdfstring{$J_n$}{Jn} on Arc Diagrams}\label{sec:action_definition}

We begin by defining the set on which $J_n$ will act. Consider the real projective line drawn as a circle. Let us define $n$ finite points on it, calling them $z_1, z_2, \dots, z_n$, as well as $z_\infty$, the point at infinity. Next, we draw some set of non-intersecting arcs connecting these points (with no restriction on how many arcs can connect any given pair). We will refer to this type of diagram as an arc diagram. 

Note that the points $z_1, z_2,\dots,z_n$ are not necessarily ordered, but we will always consider the leftmost point to be $z_\infty$. We describe this configuration using $n+1$ numbers denoted by $l_1, l_2, \dots, l_n, l_\infty$ with $l_i$ corresponding to the number of connections coming out of the point $z_i$. We also call $l_i$ the \emph{valence} of the points $z_i$. Figure~\ref{fig:simple_example} provides a simple example of an arc diagram.

\begin{figure}
    \begin{center}
        \begin{tikzpicture}

            \def\radius{1.5cm}
            \pgfmathtruncatemacro\numPoints{5} 
            
            \draw[line width=1pt] (0,0) circle (\radius);
            
            \pgfmathsetmacro\angle{360/\numPoints}
            
            \foreach \i in {1,...,\the\numexpr\numPoints-1\relax} {
                \coordinate (P\i) at (-1*\angle*\i+180:\radius);
            }
            \coordinate (P0) at (180:\radius);
            
            \draw[orange, line width=1.1pt] (P0) arc (288.0:360.0:1.0*\radius); 
            \draw[orange, line width=1.1pt] (P1) arc (216.0:288.0:1.0*\radius); 
            \draw[orange, line width=1.1pt] (P1) arc (200.39813553412063:303.6018644658794:0.75*\radius); 
            \draw[orange, line width=1.1pt] (P2) arc (115.60624956486156:172.39375043513843:2.0*\radius); 
            \draw[orange, line width=1.1pt] (P3) arc (72.0:144.0:1.0*\radius); 
            
            \foreach \i in {1,...,\the\numexpr\numPoints-1\relax} {
                \fill (P\i) circle (2pt) node[label=-1*\angle*\i+180:$z_{\i}$] {};
            }
            \fill (P0) circle (2pt) node[label=180:$z_{\infty}$] {};

            
            \end{tikzpicture}
    \end{center}
    \caption{An arc diagram for $n = 4$ with $l_1 = l_2 = 3$, $l_3 = l_\infty = 1$ and $l_4 = 2$.}\label{fig:simple_example}
\end{figure}
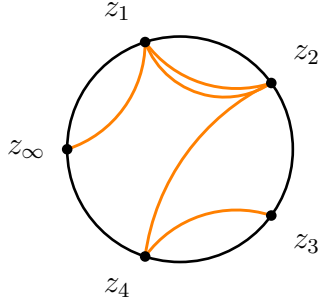

We define the set $X(l_1, l_2, \dots, l_n, l_\infty)$ as the set of all arc diagrams characterized by the given $l_i$ values. All possible configurations of $X(2, 2, 2, 2)$ can be seen in Figure~\ref{fig:X2222}. The actual set will contain 18 elements with each variant in Figure~\ref{fig:X2222} resulting in 6 elements due to relabeling of $z_1, z_2$, and $z_3$. 

\begin{figure}
    \begin{center}
        \begin{tikzpicture}
            \def\radius{1.3cm}
            \pgfmathtruncatemacro\numPoints{4} 
            \begin{scope}

            \draw[line width=1pt] (0,0) circle (\radius);
            
            \pgfmathsetmacro\angle{360/\numPoints}
            
            \foreach \i in {1,...,\the\numexpr\numPoints-1\relax} {
                \coordinate (P\i) at (-1*\angle*\i+180:\radius);
            }
            \coordinate (P0) at (180:\radius);
            
            
            \draw[orange, line width=1.1pt] (P0) arc (270.0:360.0:1.0*\radius); 
            \draw[orange, line width=1.1pt] (P0) arc (90.0:-7.105427357601002e-15:1.0*\radius); 
            \draw[orange, line width=1.1pt] (P1) arc (180.0:270.0:1.0*\radius); 
            \draw[orange, line width=1.1pt] (P2) arc (90.0:180.0:1.0*\radius); 

            \foreach \i in {1,...,\the\numexpr\numPoints-1\relax} {
                \fill (P\i) circle (2pt) node[label=-1*\angle*\i+180:$z_{i_{\i}}$] {};
            }
            \fill (P0) circle (2pt) node[label=180:$z_{\infty}$] {};
            \end{scope}

            \begin{scope}[shift={(4.5,0)}]
            
            
            \draw[line width=1pt] (0,0) circle (\radius);
            
            \pgfmathsetmacro\angle{360/\numPoints}
            
            \foreach \i in {1,...,\the\numexpr\numPoints-1\relax} {
                \coordinate (P\i) at (-1*\angle*\i+180:\radius);
            }
            \coordinate (P0) at (180:\radius);
            
            
            \draw[orange, line width=1.1pt] (P0) arc (270.0:360.0:1.0*\radius); 
            \draw[orange, line width=1.1pt] (P0) arc (244.4712206344907:385.5287793655093:0.75*\radius); 
            \draw[orange, line width=1.1pt] (P2) arc (90.0:180.0:1.0*\radius); 
            \draw[orange, line width=1.1pt] (P2) arc (64.47122063449068:205.5287793655093:0.75*\radius); 

            \foreach \i in {1,...,\the\numexpr\numPoints-1\relax} {
                \fill (P\i) circle (2pt) node[label=-1*\angle*\i+180:$z_{i_{\i}}$] {};
            }
            \fill (P0) circle (2pt) node[label=180:$z_{\infty}$] {};
            
            \end{scope}

            \begin{scope}[shift={(9,0)}]
            
                
                \draw[line width=1pt] (0,0) circle (\radius);
                
                \pgfmathsetmacro\angle{360/\numPoints}
                
                \foreach \i in {1,...,\the\numexpr\numPoints-1\relax} {
                    \coordinate (P\i) at (-1*\angle*\i+180:\radius);
                }
                \coordinate (P0) at (180:\radius);
                
                
                \draw[orange, line width=1.1pt] (P0) arc (90.0:-7.105427357601002e-15:1.0*\radius); 
                \draw[orange, line width=1.1pt] (P0) arc (115.52877936550932:-25.528779365509322:0.75*\radius); 
                \draw[orange, line width=1.1pt] (P1) arc (180.0:270.0:1.0*\radius); 
                \draw[orange, line width=1.1pt] (P1) arc (154.4712206344907:295.5287793655093:0.75*\radius); 

                \foreach \i in {1,...,\the\numexpr\numPoints-1\relax} {
                    \fill (P\i) circle (2pt) node[label=-1*\angle*\i+180:$z_{i_{\i}}$] {};
                }
                \fill (P0) circle (2pt) node[label=180:$z_{\infty}$] {};
                
            \end{scope}
            
        \end{tikzpicture}
    \end{center}
    \caption{A summary of all possible arc diagrams for $l_1 = l_2 = l_3 = l_\infty=2$ up to relabeling of points. The values $i_1$, $i_2$ and $i_3$ denote distinct natural natural numbers.}\label{fig:X2222}
\end{figure}
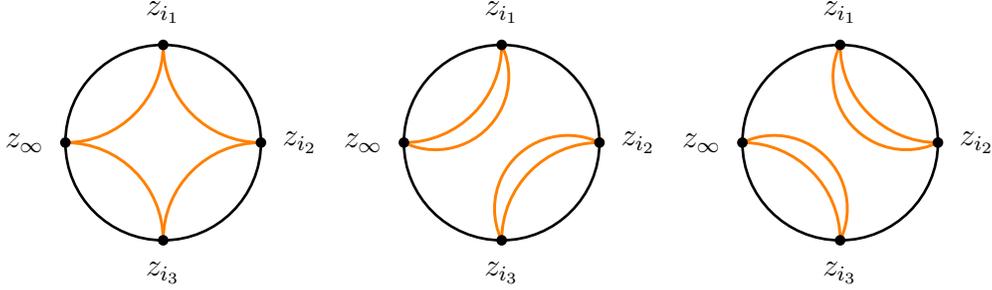

Now we define the action of $J_n$ on the set $X(l_1, l_2,\dots,l_n,l_\infty)$. We begin by assigning numbers between 1 and $n$ to the points $z_1, \dots, z_n$ starting from $z_\infty$ and going clockwise. Then $\spq$ acts by reversing the order of all points which were assigned numbers between $p$ and $q$, inclusive, and reconnecting the arcs so they remain non-intersecting. Note that the step where the arcs are reconnected is uniquely determined.

Any element $s_{p_i, q_i}\dots s_{p_2, q_2}s_{p_1, q_1} \in J_n$ acts by applying the actions corresponding to $s_{p_j, q_j}$ going in order from right to left. An example of a $J_4$ action on $X(3, 3, 1, 2, 1)$ can be seen in Figure~\ref{fig:J4_action_example}.

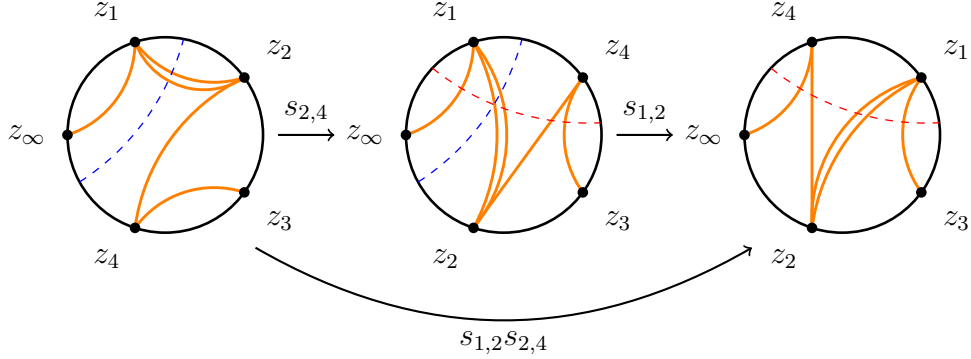
\begin{figure}
    \begin{center}
        \begin{tikzpicture}
            \def\radius{1.3cm}
            \pgfmathtruncatemacro\numPoints{5} 
            \begin{scope}

            \draw[line width=1pt] (0,0) circle (\radius);
            
            \pgfmathsetmacro\angle{360/\numPoints}
            
            \foreach \i in {1,...,\the\numexpr\numPoints-1\relax} {
                \coordinate (P\i) at (-1*\angle*\i+180:\radius);
            }
            \coordinate (P0) at (180:\radius);
            
            
            \draw[orange, line width=1.1pt] (P0) arc (288.0:360.0:1.0*\radius); 
            \draw[orange, line width=1.1pt] (P1) arc (216.0:288.0:1.0*\radius); 
            \draw[orange, line width=1.1pt] (P1) arc (200.39813553412063:303.6018644658794:0.75*\radius); 
            \draw[orange, line width=1.1pt] (P2) arc (115.60624956486156:172.39375043513843:2.0*\radius); 
            \draw[orange, line width=1.1pt] (P3) arc (72.0:144.0:1.0*\radius); 
            
            \draw[blue, dashed, line width=0.5pt] (180+0.4*\angle:\radius) arc (302.7811782102481:345.2188217897519:2.5*\radius); 

            \foreach \i in {1,...,\the\numexpr\numPoints-1\relax} {
                \fill (P\i) circle (2pt) node[label=-1*\angle*\i+180:$z_{\i}$] {};
            }
            \fill (P0) circle (2pt) node[label=180:$z_{\infty}$] {};
            \end{scope}

            \begin{scope}[shift={(4.5,0)}]
            
            
            \draw[line width=1pt] (0,0) circle (\radius);
            
            \pgfmathsetmacro\angle{360/\numPoints}
            
            \foreach \i in {1,...,\the\numexpr\numPoints-1\relax} {
                \coordinate (P\i) at (-1*\angle*\i+180:\radius);
            }
            \coordinate (P0) at (180:\radius);
            
            
            \draw[orange, line width=1.1pt] (P0) arc (288.0:360.0:1.0*\radius); 
            \draw[orange, line width=1.1pt] (P1) arc (28.393750435138447:-28.393750435138447:2.0*\radius); 
            \draw[orange, line width=1.1pt] (P1) arc (39.34864729642006:-39.34864729642006:1.5*\radius); 
            \draw[orange, line width=1.1pt] (P2) arc (144.0:216.0:1.0*\radius); 
            \draw[orange, line width=1.1pt, -] (P2) -- (P4); 

            \fill (P1) circle (2pt) node[label=-1*\angle*1+180:$z_{1}$] {};
            \fill (P2) circle (2pt) node[label=-1*\angle*2+180:$z_{4}$] {};
            \fill (P3) circle (2pt) node[label=-1*\angle*3+180:$z_{3}$] {};
            \fill (P4) circle (2pt) node[label=-1*\angle*4+180:$z_{2}$] {};
            \fill (P0) circle (2pt) node[label=180:$z_{\infty}$] {};

            \draw[blue, dashed, line width=0.5pt] (180+0.4*\angle:\radius) arc (302.7811782102481:345.2188217897519:2.5*\radius); 

            \draw[red, dashed, line width=0.5pt] (180-0.6*\angle:\radius) arc (230.78117821024807:273.2188217897519:2.5*\radius); 

            \end{scope}

            \begin{scope}[shift={(9,0)}]
            
                
                \draw[line width=1pt] (0,0) circle (\radius);
                
                \pgfmathsetmacro\angle{360/\numPoints}
                
                \foreach \i in {1,...,\the\numexpr\numPoints-1\relax} {
                    \coordinate (P\i) at (-1*\angle*\i+180:\radius);
                }
                \coordinate (P0) at (180:\radius);
                
                
                \draw[orange, line width=1.1pt] (P0) arc (288.0:360.0:1.0*\radius); 
                \draw[orange, line width=1.1pt, -] (P1) -- (P4); 
                \draw[orange, line width=1.1pt] (P2) arc (144.0:216.0:1.0*\radius); 
                \draw[orange, line width=1.1pt] (P2) arc (115.60624956486156:172.39375043513843:2.0*\radius); 
                \draw[orange, line width=1.1pt] (P2) arc (104.65135270357995:183.34864729642004:1.5*\radius); 
                
                \draw[red, dashed, line width=0.5pt] (180-0.6*\angle:\radius) arc (230.78117821024807:273.2188217897519:2.5*\radius); 
                

                \fill (P1) circle (2pt) node[label=-1*\angle*1+180:$z_{4}$] {};
                \fill (P2) circle (2pt) node[label=-1*\angle*2+180:$z_{1}$] {};
                \fill (P3) circle (2pt) node[label=-1*\angle*3+180:$z_{3}$] {};
                \fill (P4) circle (2pt) node[label=-1*\angle*4+180:$z_{2}$] {};
                \fill (P0) circle (2pt) node[label=180:$z_{\infty}$] {};
                
            \end{scope}

            \draw[thick, ->] (1.5, 0) -- (2.25, 0) node[midway, above] {{$s_{2, 4}$}};

            \draw[thick, ->] (6, 0) -- (6.75, 0) node[midway, above] {{$s_{1, 2}$}};

            \draw[thick, ->] (1.2, -1.5) to[out=-30,in=210] node[below] {$s_{1, 2} s_{2,4}$} (7.8, -1.5) ;

            
        \end{tikzpicture}
    \end{center}
    \caption{The element $s_{1,2}s_{2,4}$ acting on an element of $X(3, 3, 1, 2, 1)$.}\label{fig:J4_action_example}
\end{figure}

\begin{proposition}
    The action of $J_n$ on the set $X(l_1, l_2,\dots,l_n,l_\infty)$ described above is a group action.
\end{proposition}

\begin{proof}
    The fact that we have defined a valid action on the $z_i$'s (ignoring the connecting lines) follows directly from Proposition~\ref{prop:homomorphism_to_S}. Closure of this action on $X(l_1, l_2,\dots,l_n,l_\infty)$ follows from the fact that it preserves the values $l_1, l_2, \dots, l_n, l_\infty$. 

    To see that this action is also valid with regard to connecting lines, consider the shortest chord that sections off all $z_i$ affected by $\spq$ from the rest of the circle. We will refer to this chord as a \emph{shear line} and the region it creates on the side that does not contain $z_\infty$ as the \emph{active region}. Notice the shear lines marked in red and blue in Figure~\ref{fig:J4_action_example}. All connecting lines can be redrawn such that they cross the shear line no more than once. 
    Then, the action $\spq$ amounts to reflecting the active region over the line connecting the center of the shear line to the center of the arc bounding the active region. The broken connecting lines are then reconnected and their original order along the shear line is reversed. Nothing outside of the active region is affected.  
    Now consider the relations which define $J_n$:
    \begin{itemize}
        \item If we apply $\spq$ twice, the order along the shear line is reversed twice and the contents of the region are reflected twice over the same axis, returning to the original configuration.
        \item If we apply $\spq$ and $s_{p', q'}$ given that $[p,q] \cap [p', q'] = \emptyset$, their active regions do not intersect, therefore they can be applied in any order.
        \item If we apply $\spq s_{p', q'}\spq$ given that $[p',q'] \subset [p,q]$, consider the active region of $\spq$. When we apply $s_{p',q'}$ to the image of this region after $\spq$ is applied, this is equivalent to applying $s_{p+q-q', p+q-p'}$ to the preimage of this region and then applying $\spq$. Since $\spq^2=e$, this means that applying $\spq s_{p', q'}\spq$ is equivalent to applying $s_{p+q-q', p+q-p'}$. 
    \end{itemize}
    Thus, we have defined a valid group action. 
\end{proof}

Next, we present an important property of this action.

\begin{lemma}
    If there are two points $z_i < z_j$ such that $l_i = l_j$, they can be switched without altering the arc diagram. 
\end{lemma}
\begin{proof}
    Note that switching two adjacent points of the same valence does not affect the arc diagram. Thus, in order to satisfy the lemma, we can apply $s_{z_i + 1, z_j} s_{z_i, z_i+1} s_{z_i + 1, z_j}$, thus bringing the points together, switching them, and then applying the inverse of the operation that brought them together, thus returning the arc diagram to its original state. 
\end{proof}

This immediately leads to the following corollary:

\begin{corollary}\label{corr:order_doesn't_matter}
    Given any set of points with equal valence within an arc diagram, they can be arranged in an arbitrary order without changing any other property of the arc diagram.
\end{corollary}

In essence, this means that the specific ordering of points with the same valence does not matter in the context of counting orbits. 


\section{Simple Invariants}\label{sec:simple_invariants}

Now we consider two simple invariants of arc diagrams when acted on by $J_n$.

We define the \emph{border thickness} as the minimum number of connections between adjacent vertices in an arc diagram. Let us call the border thickness $b$. Then, the \emph{border} is the set of connecting lines consisting of the $b$ connections closest to the outside of the diagram between each adjacent pair of vertices.

\begin{lemma}\label{lemma:border_thickness}
    The border thickness in arc diagrams is an invariant under actions of $J_n$.
\end{lemma}
\begin{proof}
    When a $J_n$ action is applied, the connecting lines belonging to the border on one side of the shear line will be reconnected with the connecting lines which were part of the border on the other side, thus the border will be preserved.
\end{proof}

The second invariant is the greatest common factor of all the numbers of connecting lines between pairs of points.

\begin{lemma}\label{lemma:gcf}
    The greatest common factor of all counts of connecting lines between pairs of points is an invariant under actions of $J_n$.
\end{lemma}
\begin{proof}
    Let us call this greatest common factor $g$. We can divide all connecting lines intersecting the shear line into packets of $g$ adjacent lines which are all connected to the same point. When an action is applied, all packets are connected to other packets, so $g$ is preserved.
\end{proof}

\section{Complete Description When \texorpdfstring{$n = 3$}{n=3}}\label{sec:n=3_description}

Throughout this section, we will use the labeling shown in Figure~\ref{fig:n=3_labeling}.

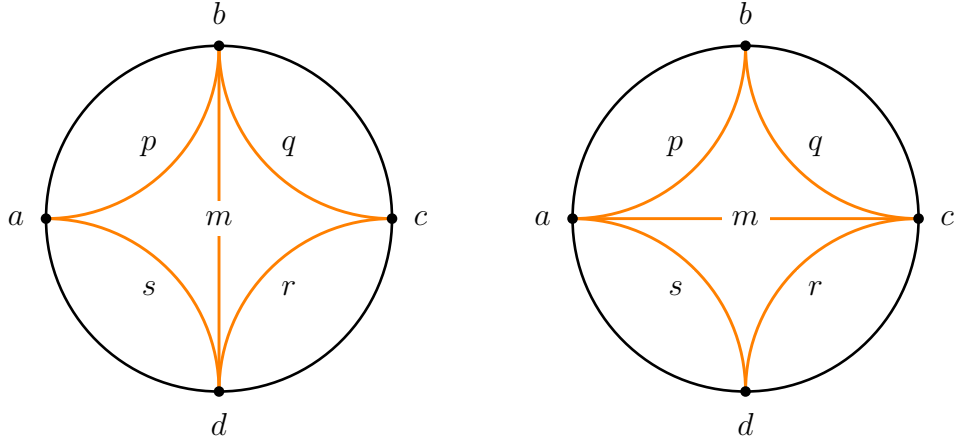
\begin{figure}
    \begin{center}
        \begin{tikzpicture}
            \def\radius{2.3cm}
            \pgfmathtruncatemacro\numPoints{4} 
            \begin{scope}

            \draw[line width=1pt] (0,0) circle (\radius);
            
            \pgfmathsetmacro\angle{360/\numPoints}
            
            \foreach \i in {1,...,\the\numexpr\numPoints-1\relax} {
                \coordinate (P\i) at (-1*\angle*\i+180:\radius);
            }
            \coordinate (P0) at (180:\radius);
            
            
            \draw[orange, line width=1.1pt] (P0) arc (270.0:360.0:1.0*\radius) node[midway, fill=white, text=black, above left] {$p$}; 
            \draw[orange, line width=1.1pt] (P0) arc (90.0:-7.105427357601002e-15:1.0*\radius) node[midway, fill=white, text=black, below left] {$s$}; 
            \draw[orange, line width=1.1pt] (P1) arc (180.0:270.0:1.0*\radius) node[midway, fill=white, text=black, above right] {$q$}; 
            \draw[orange, line width=1.1pt, -] (P1) -- (P3) node[midway, fill=white, text=black] {$m$}; 
            \draw[orange, line width=1.1pt] (P2) arc (90.0:180.0:1.0*\radius) node[midway, fill=white, text=black, below right] {$r$}; 

            \fill (P1) circle (2pt) node[label=-1*\angle*1+180:$b$] {};
            \fill (P2) circle (2pt) node[label=-1*\angle*2+180:$c$] {};
            \fill (P3) circle (2pt) node[label=-1*\angle*3+180:$d$] {};
            \fill (P0) circle (2pt) node[label=180:$a$] {};

            \end{scope}

            \begin{scope}[shift={(7,0)}]
            
            
            \draw[line width=1pt] (0,0) circle (\radius);
            
            \pgfmathsetmacro\angle{360/\numPoints}
            
            \foreach \i in {1,...,\the\numexpr\numPoints-1\relax} {
                \coordinate (P\i) at (-1*\angle*\i+180:\radius);
            }
            \coordinate (P0) at (180:\radius);
            
            
            \draw[orange, line width=1.1pt] (P0) arc (270.0:360.0:1.0*\radius) node[midway, fill=white, text=black, above left] {$p$}; 
            \draw[orange, line width=1.1pt] (P0) arc (90.0:-7.105427357601002e-15:1.0*\radius) node[midway, fill=white, text=black, below left] {$s$}; 
            \draw[orange, line width=1.1pt] (P1) arc (180.0:270.0:1.0*\radius) node[midway, fill=white, text=black, above right] {$q$}; 
            \draw[orange, line width=1.1pt, -] (P0) -- (P2) node[midway, fill=white, text=black] {$m$}; 
            \draw[orange, line width=1.1pt] (P2) arc (90.0:180.0:1.0*\radius) node[midway, fill=white, text=black, below right] {$r$}; 

            \fill (P1) circle (2pt) node[label=-1*\angle*1+180:$b$] {};
            \fill (P2) circle (2pt) node[label=-1*\angle*2+180:$c$] {};
            \fill (P3) circle (2pt) node[label=-1*\angle*3+180:$d$] {};
            \fill (P0) circle (2pt) node[label=180:$a$] {};

            \end{scope}

            
        \end{tikzpicture}
    \end{center}
    \caption{Standard labeling when $n=3$. Labels on the arcs correspond to the number of connections between the corresponding points and labels on points correspond to the valences of the points.}\label{fig:n=3_labeling}
\end{figure}

Note that whether the arc labeled $m$ will connect $b$ and $d$ or $a$ and $c$ for a given set $a,b,c,d\in \mathbb{Z}_{\geq 0}$ depends solely on whether $b + d > a + c$ or not (if they are equal, $m=0$). 

Since border thickness is an invariant by Lemma~\ref{lemma:border_thickness}, a description of the action of $J_3$ on diagrams with border thickness 0 describes the action of $J_3$ on all diagrams with $n=3$. The following lemmas allow us to describe the action of $J_3$:

\begin{lemma}\label{lemma:unique_given_hole_n=3}
    Given valence values and one of $p, q, r$ or $s = 0$ (the location of a break in the border), the arc diagram is uniquely defined.
\end{lemma}

\begin{proof}
    Without loss of generality, assume that $p=0$ and $b + d > c + a$ (the argument can be repeated analogously for any other option by permuting the variables). Then, we have the following series of linear equations:
    \begin{equation*} 
        \begin{pmatrix}
            1 & 1 & 0 & 0\\
            0 & 1 & 1 & 0\\
            1 & 0 & 1 & 1\\
            0 & 0 & 0 & 1
        \end{pmatrix}
        \begin{pmatrix}
            m\\
            q \\
            r \\
            s
        \end{pmatrix}
        =
        \begin{pmatrix}
            b\\
            c\\
            d\\
            a
        \end{pmatrix}
    \end{equation*}
    
    This matrix has determinant 2, so $m, q, r$ and $s$ are uniquely defined. 
\end{proof}

\begin{lemma}\label{lemma:only_2_diagrams}
    For a given set $a, b, c, d \in \mathbb{Z}_{\geq 0}$, there are at most 2 possible arc diagrams without a border. 
\end{lemma}

\begin{proof}
    By Lemma~\ref{lemma:unique_given_hole_n=3}, there can be at most one for each of $p, q, r$ and $s$ being 0. 

    If $b + c > a + d$, we cannot have $q = 0$ since that would mean $b + c = p + m + r \leq a + d$, which is a contradiction. Similarly, if $b + c < a + d$ we cannot have $s = 0$, if $a + b > c + d$ we cannot have $p = 0$ and if $a + b < c + d$ we cannot have $r = 0$. Note that at most two of these inequalities can be satisfied.
    
    If $a + b = c + d$, we must have $p = r$ so the cases $p = 0$ and $r = 0$ correspond to the same diagram and if $b + c = a + d$ we must have $q = s$ so the cases of $q = 0$ and $s = 0$ correspond to the same diagram. 

    Since exactly two of the cases described above are satisfied, there are no more than 2 possible arc diagrams. In other words, there is one arc diagram corresponding to the case when $p = 0$ or $r = 0$ and a second arc diagram corresponding to the case when $q = 0$ or $s = 0$ (recall that the orientation of the arc labeled $m$ is uniquely defined based on $a, b, c$ and $d$).
\end{proof}

This brings us to the description of the action of $J_3$. 

\begin{theorem}\label{thm:n=3}
    When $n=3$, the only invariant is the border thickness. In particular, all diagrams with the same border thickness lie in the same orbit.
\end{theorem}

\begin{proof}
    The fact that border thickness is an invariant is proven in Lemma~\ref{lemma:border_thickness}, so we only need to prove transitivity of the action of $J_3$ among arc diagrams with the same border thickness. Equivalently, we can prove the transitivity of the action of $J_3$ on diagrams with border thickness 0. 

    Given any two diagrams in the set $X(l_1, l_2, l_3, l_\infty)$, Proposition~\ref{prop:homomorphism_to_S} tells us that actions of $J_3$ can transform one of them so that the order of its points matches that of the other. If these diagrams are not the same, they must be the two possibilities described in Lemma~\ref{lemma:only_2_diagrams}. It remains to show that actions of $J_3$ can transform one into the other. 

    Based on casework, we can verify that the operation $s_{2,3}s_{1,3}s_{2,3}s_{1,2}$ does exactly this. An example of one such case is given in Figure~\ref{fig:n=3_proof_example}.

    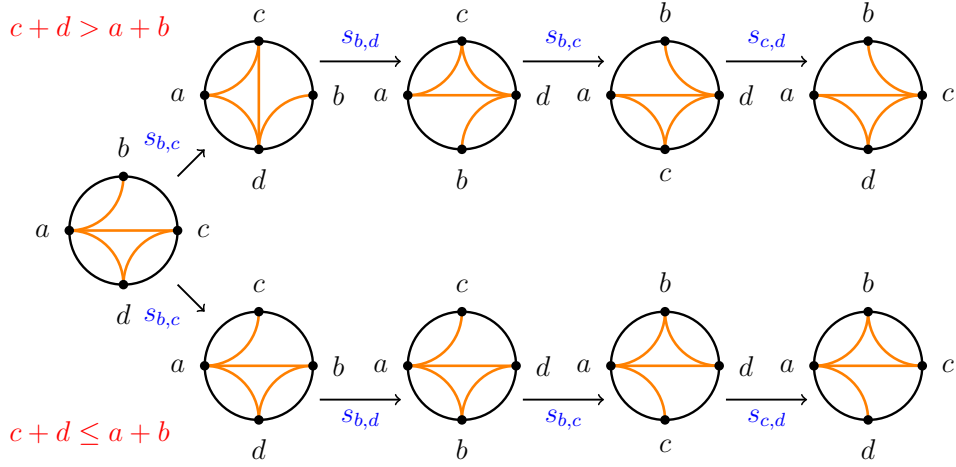
\begin{figure}
        \begin{center}
            \scalebox{0.9}{
            \begin{tikzpicture}
                \def\radius{0.8cm}
                \pgfmathtruncatemacro\numPoints{4} 
                \begin{scope}[shift={(0,0)}]

                \draw[line width=1pt] (0,0) circle (\radius);
                
                \pgfmathsetmacro\angle{360/\numPoints}
                
                \foreach \i in {1,...,\the\numexpr\numPoints-1\relax} {
                    \coordinate (P\i) at (-1*\angle*\i+180:\radius);
                }
                \coordinate (P0) at (180:\radius);
                
                
                \draw[orange, line width=1.1pt] (P0) arc (270.0:360.0:1.0*\radius); 
                \draw[orange, line width=1.1pt, -] (P0) -- (P2); 
                \draw[orange, line width=1.1pt] (P0) arc (90.0:-7.105427357601002e-15:1.0*\radius); 
                \draw[orange, line width=1.1pt] (P2) arc (90.0:180.0:1.0*\radius); 
    
                \fill (P1) circle (2pt) node[label=-1*\angle*1+180:$b$] {};
                \fill (P2) circle (2pt) node[label=-1*\angle*2+180:$c$] {};
                \fill (P3) circle (2pt) node[label=-1*\angle*3+180:$d$] {};
                \fill (P0) circle (2pt) node[label=180:$a$] {};
    
                \end{scope}

                \begin{scope}[shift={(2,-2)}]

                    \draw[line width=1pt] (0,0) circle (\radius);
                    
                    \pgfmathsetmacro\angle{360/\numPoints}
                    
                    \foreach \i in {1,...,\the\numexpr\numPoints-1\relax} {
                        \coordinate (P\i) at (-1*\angle*\i+180:\radius);
                    }
                    \coordinate (P0) at (180:\radius);
                    
                    
                    \draw[orange, line width=1.1pt] (P0) arc (270.0:360.0:1.0*\radius); 
                    \draw[orange, line width=1.1pt, -] (P0) -- (P2); 
                    \draw[orange, line width=1.1pt] (P0) arc (90.0:-7.105427357601002e-15:1.0*\radius); 
                    \draw[orange, line width=1.1pt] (P2) arc (90.0:180.0:1.0*\radius); 
        
                    \fill (P1) circle (2pt) node[label=-1*\angle*1+180:$c$] {};
                    \fill (P2) circle (2pt) node[label=-1*\angle*2+180:$b$] {};
                    \fill (P3) circle (2pt) node[label=-1*\angle*3+180:$d$] {};
                    \fill (P0) circle (2pt) node[label=180:$a$] {};
        
                \end{scope}

                \begin{scope}[shift={(5,-2)}]

                    \draw[line width=1pt] (0,0) circle (\radius);
                    
                    \pgfmathsetmacro\angle{360/\numPoints}
                    
                    \foreach \i in {1,...,\the\numexpr\numPoints-1\relax} {
                        \coordinate (P\i) at (-1*\angle*\i+180:\radius);
                    }
                    \coordinate (P0) at (180:\radius);
                    
                    
                    \draw[orange, line width=1.1pt] (P0) arc (270.0:360.0:1.0*\radius); 
                    \draw[orange, line width=1.1pt, -] (P0) -- (P2); 
                    \draw[orange, line width=1.1pt] (P0) arc (90.0:-7.105427357601002e-15:1.0*\radius); 
                    \draw[orange, line width=1.1pt] (P2) arc (90.0:180.0:1.0*\radius); 
        
                    \fill (P1) circle (2pt) node[label=-1*\angle*1+180:$c$] {};
                    \fill (P2) circle (2pt) node[label=-1*\angle*2+180:$d$] {};
                    \fill (P3) circle (2pt) node[label=-1*\angle*3+180:$b$] {};
                    \fill (P0) circle (2pt) node[label=180:$a$] {};
        
                \end{scope}

                \begin{scope}[shift={(8,-2)}]

                    \draw[line width=1pt] (0,0) circle (\radius);
                    
                    \pgfmathsetmacro\angle{360/\numPoints}
                    
                    \foreach \i in {1,...,\the\numexpr\numPoints-1\relax} {
                        \coordinate (P\i) at (-1*\angle*\i+180:\radius);
                    }
                    \coordinate (P0) at (180:\radius);
                    
                    
                    \draw[orange, line width=1.1pt] (P0) arc (270.0:360.0:1.0*\radius); 
                    \draw[orange, line width=1.1pt, -] (P0) -- (P2); 
                    \draw[orange, line width=1.1pt] (P0) arc (90.0:-7.105427357601002e-15:1.0*\radius); 
                    \draw[orange, line width=1.1pt] (P1) arc (180.0:270.0:1.0*\radius); 
                            
                    \fill (P1) circle (2pt) node[label=-1*\angle*1+180:$b$] {};
                    \fill (P2) circle (2pt) node[label=-1*\angle*2+180:$d$] {};
                    \fill (P3) circle (2pt) node[label=-1*\angle*3+180:$c$] {};
                    \fill (P0) circle (2pt) node[label=180:$a$] {};
        
                \end{scope}

                \begin{scope}[shift={(11,-2)}]

                    \draw[line width=1pt] (0,0) circle (\radius);
                    
                    \pgfmathsetmacro\angle{360/\numPoints}
                    
                    \foreach \i in {1,...,\the\numexpr\numPoints-1\relax} {
                        \coordinate (P\i) at (-1*\angle*\i+180:\radius);
                    }
                    \coordinate (P0) at (180:\radius);
                    
                    
                    \draw[orange, line width=1.1pt] (P0) arc (270.0:360.0:1.0*\radius); 
                    \draw[orange, line width=1.1pt, -] (P0) -- (P2); 
                    \draw[orange, line width=1.1pt] (P0) arc (90.0:-7.105427357601002e-15:1.0*\radius); 
                    \draw[orange, line width=1.1pt] (P1) arc (180.0:270.0:1.0*\radius); 
                            
                    \fill (P1) circle (2pt) node[label=-1*\angle*1+180:$b$] {};
                    \fill (P2) circle (2pt) node[label=-1*\angle*2+180:$c$] {};
                    \fill (P3) circle (2pt) node[label=-1*\angle*3+180:$d$] {};
                    \fill (P0) circle (2pt) node[label=180:$a$] {};
        
                \end{scope}

                \begin{scope}[shift={(2,2)}]

                    \draw[line width=1pt] (0,0) circle (\radius);
                    
                    \pgfmathsetmacro\angle{360/\numPoints}
                    
                    \foreach \i in {1,...,\the\numexpr\numPoints-1\relax} {
                        \coordinate (P\i) at (-1*\angle*\i+180:\radius);
                    }
                    \coordinate (P0) at (180:\radius);
                    
                    
                    \draw[orange, line width=1.1pt] (P0) arc (270.0:360.0:1.0*\radius); 
                    \draw[orange, line width=1.1pt] (P0) arc (90.0:-7.105427357601002e-15:1.0*\radius); 
                    \draw[orange, line width=1.1pt, -] (P1) -- (P3); 
                    \draw[orange, line width=1.1pt] (P2) arc (90.0:180.0:1.0*\radius); 
                            
                    \fill (P1) circle (2pt) node[label=-1*\angle*1+180:$c$] {};
                    \fill (P2) circle (2pt) node[label=-1*\angle*2+180:$b$] {};
                    \fill (P3) circle (2pt) node[label=-1*\angle*3+180:$d$] {};
                    \fill (P0) circle (2pt) node[label=180:$a$] {};
        
                \end{scope}

                \begin{scope}[shift={(5,2)}]

                    \draw[line width=1pt] (0,0) circle (\radius);
                    
                    \pgfmathsetmacro\angle{360/\numPoints}
                    
                    \foreach \i in {1,...,\the\numexpr\numPoints-1\relax} {
                        \coordinate (P\i) at (-1*\angle*\i+180:\radius);
                    }
                    \coordinate (P0) at (180:\radius);
                    
                    
                    \draw[orange, line width=1.1pt] (P0) arc (270.0:360.0:1.0*\radius); 
                    \draw[orange, line width=1.1pt, -] (P0) -- (P2); 
                    \draw[orange, line width=1.1pt] (P1) arc (180.0:270.0:1.0*\radius); 
                    \draw[orange, line width=1.1pt] (P2) arc (90.0:180.0:1.0*\radius); 
                            
                    \fill (P1) circle (2pt) node[label=-1*\angle*1+180:$c$] {};
                    \fill (P2) circle (2pt) node[label=-1*\angle*2+180:$d$] {};
                    \fill (P3) circle (2pt) node[label=-1*\angle*3+180:$b$] {};
                    \fill (P0) circle (2pt) node[label=180:$a$] {};
        
                \end{scope}
 
                \begin{scope}[shift={(8,2)}]

                    \draw[line width=1pt] (0,0) circle (\radius);
                    
                    \pgfmathsetmacro\angle{360/\numPoints}
                    
                    \foreach \i in {1,...,\the\numexpr\numPoints-1\relax} {
                        \coordinate (P\i) at (-1*\angle*\i+180:\radius);
                    }
                    \coordinate (P0) at (180:\radius);
                    
                    
                    \draw[orange, line width=1.1pt, -] (P0) -- (P2); 
                    \draw[orange, line width=1.1pt] (P0) arc (90.0:-7.105427357601002e-15:1.0*\radius); 
                    \draw[orange, line width=1.1pt] (P1) arc (180.0:270.0:1.0*\radius); 
                    \draw[orange, line width=1.1pt] (P2) arc (90.0:180.0:1.0*\radius); 
                            
                    \fill (P1) circle (2pt) node[label=-1*\angle*1+180:$b$] {};
                    \fill (P2) circle (2pt) node[label=-1*\angle*2+180:$d$] {};
                    \fill (P3) circle (2pt) node[label=-1*\angle*3+180:$c$] {};
                    \fill (P0) circle (2pt) node[label=180:$a$] {};
        
                \end{scope}
 
                \begin{scope}[shift={(11,2)}]

                    \draw[line width=1pt] (0,0) circle (\radius);
                    
                    \pgfmathsetmacro\angle{360/\numPoints}
                    
                    \foreach \i in {1,...,\the\numexpr\numPoints-1\relax} {
                        \coordinate (P\i) at (-1*\angle*\i+180:\radius);
                    }
                    \coordinate (P0) at (180:\radius);
                    
                    
                    \draw[orange, line width=1.1pt, -] (P0) -- (P2); 
                    \draw[orange, line width=1.1pt] (P0) arc (90.0:-7.105427357601002e-15:1.0*\radius); 
                    \draw[orange, line width=1.1pt] (P1) arc (180.0:270.0:1.0*\radius); 
                    \draw[orange, line width=1.1pt] (P2) arc (90.0:180.0:1.0*\radius); 
                            
                    \fill (P1) circle (2pt) node[label=-1*\angle*1+180:$b$] {};
                    \fill (P2) circle (2pt) node[label=-1*\angle*2+180:$c$] {};
                    \fill (P3) circle (2pt) node[label=-1*\angle*3+180:$d$] {};
                    \fill (P0) circle (2pt) node[label=180:$a$] {};
        
                \end{scope}

                \draw[thick, ->] (0.8, 0.8) -- (1.2, 1.2)  node[midway, above left, text=blue] {$s_{b,c}$};
                \draw[thick, ->] (0.8, -0.8) -- (1.2, -1.2)  node[midway, below left, text=blue] {$s_{b,c}$};

                \draw[thick, ->] (2.9, 2.5) -- (4.1, 2.5)  node[midway, above, text=blue] {$s_{b,d}$};
                \draw[thick, ->] (2.9, -2.5) -- (4.1, -2.5)  node[midway, below, text=blue] {$s_{b,d}$};

                \draw[thick, ->] (5.9, 2.5) -- (7.1, 2.5)  node[midway, above, text=blue] {$s_{b,c}$};
                \draw[thick, ->] (5.9, -2.5) -- (7.1, -2.5)  node[midway, below, text=blue] {$s_{b,c}$};

                \draw[thick, ->] (8.9, 2.5) -- (10.1, 2.5)  node[midway, above, text=blue] {$s_{c,d}$};
                \draw[thick, ->] (8.9, -2.5) -- (10.1, -2.5)  node[midway, below, text=blue] {$s_{c,d}$};

                \draw (-0.5, 3) node[text=red] {$c+d > a + b$};
                \draw (-0.5, -3) node[text=red] {$c+d \leq a + b$};

                
            \end{tikzpicture}
            }
        \end{center}
        \caption{An example of one of the cases in Theorem~\ref{thm:n=3}. The arcs are unlabeled but they are uniquely determined based on Lemma~\ref{lemma:unique_given_hole_n=3}.}\label{fig:n=3_proof_example}
    \end{figure}
\end{proof}

\section{Transitivity When \texorpdfstring{$l_i = 1$}{li=1}}\label{sec:transitivity1}

We consider the actions of $J_n$ on a set $X(l_1, l_2, \dots, l_n, l_\infty)$ when there exists some $l_i = 1$ (or $l_\infty = 1$).

First, note that if we wish to perform an operation that reverses points $q$ through $p$, where $q > p$ (that is, the interval includes $z_\infty$), we instead apply $s_{p+1,q-1}$, achieving the same result (unless $q = p + 1$, in which case reversing points $q$ through $p$ does not do anything). Thus, we can treat $s_\infty$ as we would any other point on the circle. 

Also, we will use $z_k$ to mean either the point $z_k$ or the index of the point $z_k$ when counting from $z_\infty$, depending on the context. In particular, this allows us to use $s_{z_k, z_j}$, where $z_k$ and $z_j$ are points, to mean reversing the points between $z_k$ and $z_j$, inclusive. We will also use $y_k$ to mean the point at index $k$, regardless of how the $z_j$'s are arranged.

Moving forward, we define $i=$ such that $l_i = 1$. Let us define the \emph{adjacency operation}, which brings the point $z_i$ to be adjacent to the point to which it is connected while keeping the rest of the diagram unchanged. Let the initial index of $z_i$ be $j$ and the index of the point to which it is connected be $k$. If $j < k$ then the action $s_{j, k-2} s_{j, k-1}$ does this and if $j > k$ then the action $s_{k+1, j} s_{k+2, j}$ does this. 

Next, we define the \emph{fissure operation}. We assume that $z_i$ is connected to a point immediately adjacent to it. We define the fissure operation $f_{j, k}$, where $j$ is the index of $z_i$ and $k$ is any other index, as $f^-_{j, k} = s_{j, k-1} s_{j, k} $ if $z_i$ is connected to the point adjacent to it in the negative direction and $f^+_{j,k} = s_{j, k} s_{j + 1, k}$ otherwise. 

Note that we can change which side $z_i$ is connected on by doing the action of the cactus group that switches it with the point it is connected to without affecting any other connections.

Graphically, we can interpret this operation as creating a ``fissure'' connecting the space between $z_i$ and the point to which it is initially connected and the space on the positive side of $y_k$ and shifting all connecting lines intersected by this fissure on the side containing $z_i$ toward $z_i$ by 1. The point $z_i$ ends up at index $k$, connected to the last connecting line intersected by the fissure. For example, see Figure~\ref{fig:fissure_example}. 

Note that the order of the points other than $z_i$ remains unchanged and all connections not intersected by the fissure remain unchanged. 

\begin{figure}
    \begin{center}
        \scalebox{0.8}{
        \begin{tikzpicture}
            \def\radius{2.3cm}
            \pgfmathtruncatemacro\numPoints{7} 
            \begin{scope}

            \draw[line width=1pt] (0,0) circle (\radius);
            
            \pgfmathsetmacro\angle{360/\numPoints}
            
            \foreach \i in {1,...,\the\numexpr\numPoints-1\relax} {
                \coordinate (P\i) at (-1*\angle*\i+180:\radius);
            }
            \coordinate (P0) at (180:\radius);
            
            
            \draw[orange, line width=1.1pt] (P0) arc (308.57142857142856:360.0:1.0*\radius); 
            \draw[orange, line width=1.1pt] (P0) arc (285.55993707149213:331.582920071365:2.0*\radius); 
            \draw[orange, line width=1.1pt] (P0) arc (277.157239903014:339.9856172398431:1.5*\radius); 
            \draw[orange, line width=1.1pt] (P0) arc (263.8930236676988:301.821262046587:3.0*\radius); 
            \draw[orange, line width=1.1pt] (P2) arc (205.71428571428572:257.14285714285717:1.0*\radius); 
            \draw[orange, line width=1.1pt] (P3) arc (154.28571428571428:205.71428571428572:1.0*\radius); 
            \draw[orange, line width=1.1pt] (P3) arc (144.65407187131078:215.34592812868922:0.75*\radius); 
            \draw[orange, line width=1.1pt] (P4) arc (102.85714285714283:154.28571428571428:1.0*\radius); 
            \draw[orange, line width=1.1pt] (P4) arc (79.84565135720644:125.86863435707929:2.0*\radius); 

            \draw[red, dashed, line width=0.5pt] (180-0.6*\angle:\radius) arc (219.99743913340671:242.85970372373612:5*\radius); 
                        
            \foreach \i in {1,...,\the\numexpr\numPoints-1\relax} {
                \fill (P\i) circle (2pt);
            }
            \fill (P1) circle (2pt) node[label=-1*\angle*1+180:$z_i$] {};
            \fill (P2) circle (2pt) node[label=-1*\angle*2+180:$z_1$] {};
            \fill (P3) circle (2pt) node[label=-1*\angle*3+180:$z_2$] {};
            \fill (P4) circle (2pt) node[label=-1*\angle*4+180:$z_3$] {};
            \fill (P5) circle (2pt) node[label=-1*\angle*5+180:$z_4$] {};
            \fill (P6) circle (2pt) node[label=-1*\angle*6+180:$z_5$] {};
            \fill (P0) circle (2pt) node[label=180:$z_\infty$] {};

            \end{scope}

            \begin{scope}[shift={(7.5,0)}]
            
            
            \draw[line width=1pt] (0,0) circle (\radius);
            
            \pgfmathsetmacro\angle{360/\numPoints}
            
            \foreach \i in {1,...,\the\numexpr\numPoints-1\relax} {
                \coordinate (P\i) at (-1*\angle*\i+180:\radius);
            }
            \coordinate (P0) at (180:\radius);
            
            
            \draw[orange, line width=1.1pt] (P0) arc (308.57142857142856:360.0:1.0*\radius); 
            \draw[orange, line width=1.1pt] (P0) arc (298.9397861570251:369.6316424144035:0.75*\radius); 
            \draw[orange, line width=1.1pt] (P0) arc (285.55993707149213:331.582920071365:2.0*\radius); 
            \draw[orange, line width=1.1pt] (P0) arc (263.8930236676988:301.821262046587:3.0*\radius); 
            \draw[orange, line width=1.1pt] (P1) arc (257.1428571428571:308.57142857142856:1.0*\radius); 
            \draw[orange, line width=1.1pt] (P2) arc (205.71428571428572:257.14285714285717:1.0*\radius); 
            \draw[orange, line width=1.1pt] (P2) arc (196.08264329988222:266.77449955726064:0.75*\radius); 
            \draw[orange, line width=1.1pt] (P3) -- (P6); 
            \draw[orange, line width=1.1pt] (P4) arc (102.85714285714283:154.28571428571428:1.0*\radius); 
            
            \draw[red, dashed, line width=0.5pt] (180-0.6*\angle:\radius) arc (219.99743913340671:242.85970372373612:5*\radius); 
            
            \fill (P4) circle (2pt) node[label=-4*\angle*1+180:$z_i$] {};
            \fill (P1) circle (2pt) node[label=-1*\angle*1+180:$z_1$] {};
            \fill (P2) circle (2pt) node[label=-1*\angle*2+180:$z_2$] {};
            \fill (P3) circle (2pt) node[label=-1*\angle*3+180:$z_3$] {};
            \fill (P5) circle (2pt) node[label=-1*\angle*5+180:$z_4$] {};
            \fill (P6) circle (2pt) node[label=-1*\angle*6+180:$z_5$] {};
            \fill (P0) circle (2pt) node[label=180:$z_\infty$] {};

            \end{scope}

            \draw[thick, ->] (2.7, 0) -- (4.2, 0)  node[midway, above, text=blue] {$s_{1,3} s_{1,4}$};

            \begin{scope}[shift={(0, -6)}]

                \draw[line width=1pt] (0,0) circle (\radius);
                
                \pgfmathsetmacro\angle{360/\numPoints}
                
                \foreach \i in {1,...,\the\numexpr\numPoints-1\relax} {
                    \coordinate (P\i) at (-1*\angle*\i+180:\radius);
                }
                \coordinate (P0) at (180:\radius);
                
                
                \draw[orange, line width=1.1pt] (P0) arc (285.55993707149213:331.582920071365:2.0*\radius); 
                \draw[orange, line width=1.1pt] (P0) arc (277.157239903014:339.9856172398431:1.5*\radius); 
                \draw[orange, line width=1.1pt] (P0) arc (263.8930236676988:301.821262046587:3.0*\radius); 
                \draw[orange, line width=1.1pt] (P1) arc (257.1428571428571:308.57142857142856:1.0*\radius); 
                \draw[orange, line width=1.1pt] (P2) arc (205.71428571428572:257.14285714285717:1.0*\radius); 
                \draw[orange, line width=1.1pt] (P3) arc (154.28571428571428:205.71428571428572:1.0*\radius); 
                \draw[orange, line width=1.1pt] (P3) arc (144.65407187131078:215.34592812868922:0.75*\radius); 
                \draw[orange, line width=1.1pt] (P4) arc (102.85714285714283:154.28571428571428:1.0*\radius); 
                \draw[orange, line width=1.1pt] (P4) arc (79.84565135720644:125.86863435707929:2.0*\radius); 
    
                \draw[red, dashed, line width=0.5pt] (180-1.6*\angle:\radius) arc (168.56886770483533:191.43113229516467:5*\radius); 

                \fill (P1) circle (2pt) node[text=red, label=-1*\angle*1+180:$z_i$] {};
                \fill (P2) circle (2pt) node[label=-1*\angle*2+180:$z_1$] {};
                \fill (P3) circle (2pt) node[label=-1*\angle*3+180:$z_2$] {};
                \fill (P4) circle (2pt) node[label=-1*\angle*4+180:$z_3$] {};
                \fill (P5) circle (2pt) node[label=-1*\angle*5+180:$z_4$] {};
                \fill (P6) circle (2pt) node[label=-1*\angle*6+180:$z_5$] {};

                \fill (P0) circle (2pt) node[label=180:$z_\infty$] {};
    
                \end{scope}
    
                \begin{scope}[shift={(7.5,-6)}]
                
                
                \draw[line width=1pt] (0,0) circle (\radius);
                
                \pgfmathsetmacro\angle{360/\numPoints}
                
                \foreach \i in {1,...,\the\numexpr\numPoints-1\relax} {
                    \coordinate (P\i) at (-1*\angle*\i+180:\radius);
                }
                \coordinate (P0) at (180:\radius);
                
                
                \draw[orange, line width=1.1pt] (P0) arc (308.57142857142856:360.0:1.0*\radius); 
                \draw[orange, line width=1.1pt] (P0) arc (298.9397861570251:369.6316424144035:0.75*\radius); 
                \draw[orange, line width=1.1pt] (P0) arc (292.1273073538059:376.44412121762264:0.6464466094067263*\radius); 
                \draw[orange, line width=1.1pt] (P1) arc (257.1428571428571:308.57142857142856:1.0*\radius); 
                \draw[orange, line width=1.1pt] (P2) arc (205.71428571428572:257.14285714285717:1.0*\radius); 
                \draw[orange, line width=1.1pt] (P2) arc (196.08264329988222:266.77449955726064:0.75*\radius); 
                \draw[orange, line width=1.1pt] (P2) arc (-6.750166524841621:-44.6784049037298:3.0*\radius); 
                \draw[orange, line width=1.1pt] (P3) arc (154.28571428571428:205.71428571428572:1.0*\radius); 
                \draw[orange, line width=1.1pt] (P3) arc (131.27422278577785:177.2972057856507:2.0*\radius); 
                
                \draw[red, dashed, line width=0.5pt] (180-0.6*\angle:\radius) arc (219.99743913340671:242.85970372373612:5*\radius); 

                \fill (P5) circle (2pt) node[label=-5*\angle*1+180:$z_i$] {};
                \fill (P1) circle (2pt) node[label=-1*\angle*1+180:$z_1$] {};
                \fill (P2) circle (2pt) node[label=-1*\angle*2+180:$z_2$] {};
                \fill (P3) circle (2pt) node[label=-1*\angle*3+180:$z_3$] {};
                \fill (P4) circle (2pt) node[label=-1*\angle*4+180:$z_4$] {};
                \fill (P6) circle (2pt) node[label=-1*\angle*6+180:$z_5$] {};
                \fill (P0) circle (2pt) node[label=180:$z_\infty$] {};

                \end{scope}
    
                \draw[thick, ->] (2.7, -6) -- (4.2, -6)  node[midway, above, text=blue] {$s_{1,5} s_{2,5}$};

            
        \end{tikzpicture}
        }
    \end{center}
    \caption{An example of the fissure operations $f^-_{1, 4}$ and $f^+_{1,5}$.}\label{fig:fissure_example}
\end{figure}
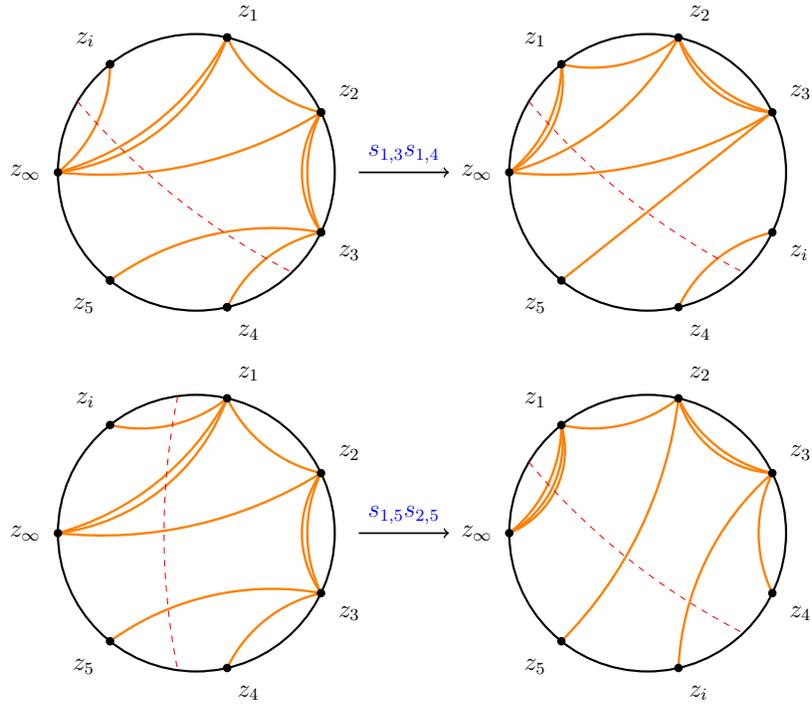

The following two lemmas allow us to completely describe the $J_n$ orbits in the set $X(l_1, l_2, \dots, l_n, l_\infty)$ where $l_i = 1$. 

\begin{lemma}\label{lemma:remove_connection}
    Consider an initial arc diagram $x \in X(l_1, l_2,\dots,l_n, l_\infty)$ with $n \geq 3$. Let us select an arbitrary pair of neighboring points $z_k < z_m$, neither of which is connected to $z_i$. 
    
    Assume there exists $y \in X(l_1, l_2,\dots,l_n, l_\infty)$ where the number of connections between $z_k$ and $z_m$ is smaller than in $x$ and the points $z_k$ and $z_m$ are not connected to $z_i$. Then there exists $j$ in the cactus group such that $jx$ has one less connection between $z_k$ and $z_m$ and neither point is connected to $z_i$. Furthermore, the order of the points other than $z_i$ in $jx$ is the same as in $x$. 
\end{lemma}

\begin{proof}
    We begin by performing $f^-_{z_i, z_k}$, creating an arc diagram with one less connecting line between $z_k$ and $z_m$. Now, we must find a way to disconnect $z_i$ from $z_m$ without adding a connecting line between $z_k$ and $z_m$. 
    
    First, we apply $s_{z_i, z_m}$. Next, find the closest point $y_{z_i + q}$ for positive $q$, that is the closest point in the positive direction from $z_i$, such that it satisfies one of three conditions: it is not connected to $z_k$ or $z_m$, it is connected to $z_m$ and some other point that is not $z_k$, or it is connected only to $z_m$ and $z_k$. If one of the first two cases is satisfied, apply $f^-_{z_i, z_i + q}$ and we are done. See Figure~\ref{fig:transitivity_example_remove} for an example of this case.

    \begin{figure}
        \begin{center}
            \begin{tikzpicture}
                \def\radius{1cm}
                \pgfmathtruncatemacro\numPoints{7} 
                \begin{scope}[shift={(0,0)}]

                \draw[line width=1pt] (0,0) circle (\radius);
                
                \pgfmathsetmacro\angle{360/\numPoints}
                
                \foreach \i in {1,...,\the\numexpr\numPoints-1\relax} {
                    \coordinate (P\i) at (-1*\angle*\i+180:\radius);
                }
                \coordinate (P0) at (180:\radius);
                
                
                \draw[orange, line width=1.1pt] (P0) arc (308.57142857142856:360.0:1.0*\radius); 
                \draw[orange, line width=1.1pt] (P0) arc (263.8930236676988:301.821262046587:3.0*\radius); 
                \draw[orange, line width=1.1pt] (P0) -- (P3); 
                \draw[orange, line width=1.1pt] (P0) arc (51.428571428571445:0.0:1.0*\radius); 
                \draw[orange, line width=1.1pt] (P1) arc (257.1428571428571:308.57142857142856:1.0*\radius); 
                \draw[orange, line width=1.1pt] (P3) -- (P6); 
                \draw[orange, line width=1.1pt] (P4) arc (102.85714285714283:154.28571428571428:1.0*\radius); 
                \draw[orange, line width=1.1pt] (P4) arc (86.69598842164707:170.44686872121002:0.65*\radius); 
                \draw[orange, line width=1.1pt] (P4) arc (69.35208043935259:187.79077670350452:0.5050252531694167*\radius); 
                \draw[orange, line width=1.1pt] (P5) arc (51.42857142857142:102.85714285714286:1.0*\radius); 

                \draw[red, dashed, line width=0.5pt] (180-1.6*\angle:\radius) arc (194.749169845637:216.6794015829344:5*\radius); 


                \foreach \i in {0,...,\the\numexpr\numPoints-1\relax} {
                    \fill (P\i) circle (2pt);
                }
                
                \fill (P2) circle (2pt) node[label=-1*\angle*2+180:$z_i$] {};
                \fill (P4) circle (2pt) node[label=-1*\angle*4+180:$z_k$] {};
                \fill (P5) circle (2pt) node[label=-1*\angle*5+180:$z_m$] {};
    
                \end{scope}

                \begin{scope}[shift={(3.3,0)}]

                    \draw[line width=1pt] (0,0) circle (\radius);
                    
                    \pgfmathsetmacro\angle{360/\numPoints}
                    
                    \foreach \i in {1,...,\the\numexpr\numPoints-1\relax} {
                        \coordinate (P\i) at (-1*\angle*\i+180:\radius);
                    }
                    \coordinate (P0) at (180:\radius);
                    
                    
                    \draw[orange, line width=1.1pt] (P0) arc (308.57142857142856:360.0:1.0*\radius); 
\draw[orange, line width=1.1pt] (P0) arc (285.55993707149213:331.582920071365:2.0*\radius); 
\draw[orange, line width=1.1pt] (P0) arc (271.60056394422736:345.54229319862975:1.3*\radius); 
\draw[orange, line width=1.1pt] (P0) arc (51.428571428571445:0.0:1.0*\radius); 
\draw[orange, line width=1.1pt] (P1) arc (257.1428571428571:308.57142857142856:1.0*\radius); 
\draw[orange, line width=1.1pt] (P3) arc (131.27422278577785:177.2972057856507:2.0*\radius); 
\draw[orange, line width=1.1pt] (P3) arc (117.31484965851308:191.25657891291547:1.3*\radius); 
\draw[orange, line width=1.1pt] (P3) arc (109.60730938198449:147.53554776087267:3.0*\radius); 
\draw[orange, line width=1.1pt] (P4) arc (102.85714285714283:154.28571428571428:1.0*\radius); 
\draw[orange, line width=1.1pt] (P5) arc (51.42857142857142:102.85714285714286:1.0*\radius); 

\draw[red, dashed, line width=0.5pt] (180-3.6*\angle:\radius) arc (107.38538664446662:149.7574704983905:2*\radius); 


                    \foreach \i in {0,...,\the\numexpr\numPoints-1\relax} {
                        \fill (P\i) circle (2pt);
                    }
                    \fill (P3) circle (2pt) node[label=-1*\angle*3+180:$z_k$] {};
                    \fill (P4) circle (2pt) node[label=-1*\angle*4+180:$z_i$] {};
                    \fill (P5) circle (2pt) node[label=-1*\angle*5+180:$z_m$] {};
        
                \end{scope}

                \begin{scope}[shift={(6.6,0)}]

                    \draw[line width=1pt] (0,0) circle (\radius);
                    
                    \pgfmathsetmacro\angle{360/\numPoints}
                    
                    \foreach \i in {1,...,\the\numexpr\numPoints-1\relax} {
                        \coordinate (P\i) at (-1*\angle*\i+180:\radius);
                    }
                    \coordinate (P0) at (180:\radius);
                    
                    
                    \draw[orange, line width=1.1pt] (P0) arc (308.57142857142856:360.0:1.0*\radius); 
\draw[orange, line width=1.1pt] (P0) arc (285.55993707149213:331.582920071365:2.0*\radius); 
\draw[orange, line width=1.1pt] (P0) arc (271.60056394422736:345.54229319862975:1.3*\radius); 
\draw[orange, line width=1.1pt] (P0) arc (51.428571428571445:0.0:1.0*\radius); 
\draw[orange, line width=1.1pt] (P1) arc (257.1428571428571:308.57142857142856:1.0*\radius); 
\draw[orange, line width=1.1pt] (P3) arc (154.28571428571428:205.71428571428572:1.0*\radius); 
\draw[orange, line width=1.1pt] (P3) arc (138.12455985021853:221.87544014978147:0.65*\radius); 
\draw[orange, line width=1.1pt] (P3) arc (109.60730938198449:147.53554776087267:3.0*\radius); 
\draw[orange, line width=1.1pt] (P4) arc (102.85714285714283:154.28571428571428:1.0*\radius); 
\draw[orange, line width=1.1pt] (P4) arc (79.84565135720644:125.86863435707929:2.0*\radius); 

\draw[red, dashed, line width=0.5pt] (180-4.6*\angle:\radius) arc (55.95681521589521:98.3288990698191:2*\radius); 
        
                    \foreach \i in {0,...,\the\numexpr\numPoints-1\relax} {
                        \fill (P\i) circle (2pt);
                    }
                    \fill (P3) circle (2pt) node[label=-1*\angle*3+180:$z_k$] {};
                    \fill (P4) circle (2pt) node[label=-1*\angle*4+180:$z_m$] {};
                    \fill (P5) circle (2pt) node[label=-1*\angle*5+180:$z_i$] {};
        
                \end{scope}

                \begin{scope}[shift={(9.9,0)}]

                    \draw[line width=1pt] (0,0) circle (\radius);
                    
                    \pgfmathsetmacro\angle{360/\numPoints}
                    
                    \foreach \i in {1,...,\the\numexpr\numPoints-1\relax} {
                        \coordinate (P\i) at (-1*\angle*\i+180:\radius);
                    }
                    \coordinate (P0) at (180:\radius);
                    
                    
                    \draw[orange, line width=1.1pt] (P0) arc (308.57142857142856:360.0:1.0*\radius); 
\draw[orange, line width=1.1pt] (P0) arc (285.55993707149213:331.582920071365:2.0*\radius); 
\draw[orange, line width=1.1pt] (P0) arc (271.60056394422736:345.54229319862975:1.3*\radius); 
\draw[orange, line width=1.1pt] (P0) arc (51.428571428571445:0.0:1.0*\radius); 
\draw[orange, line width=1.1pt] (P1) arc (257.1428571428571:308.57142857142856:1.0*\radius); 
\draw[orange, line width=1.1pt] (P3) arc (154.28571428571428:205.71428571428572:1.0*\radius); 
\draw[orange, line width=1.1pt] (P3) arc (138.12455985021853:221.87544014978147:0.65*\radius); 
\draw[orange, line width=1.1pt] (P3) arc (131.27422278577785:177.2972057856507:2.0*\radius); 
\draw[orange, line width=1.1pt] (P4) arc (102.85714285714283:154.28571428571428:1.0*\radius); 
\draw[orange, line width=1.1pt] (P4) arc (86.69598842164707:170.44686872121002:0.65*\radius); 
                            
                    
                    \foreach \i in {0,...,\the\numexpr\numPoints-1\relax} {
                        \fill (P\i) circle (2pt);
                    }
                    
                    \fill (P3) circle (2pt) node[label=-1*\angle*3+180:$z_k$] {};
                    \fill (P4) circle (2pt) node[label=-1*\angle*4+180:$z_m$] {};
                    \fill (P6) circle (2pt) node[label=-1*\angle*6+180:$z_i$] {};
        
                \end{scope}

                \draw[thick, ->] (1.2, 0) -- (2, 0)  node[midway, above, text=blue] {$f^-$};

                \draw[thick, ->] (1.2+3.3, 0) -- (2+3.3, 0)  node[midway, above, text=blue] {$s$};

                \draw[thick, ->] (1.2+3.3*2, 0) -- (2+3.3*2, 0)  node[midway, above, text=blue] {$f^-$};


            \end{tikzpicture}
        \end{center}
        \caption{A simple example of actions required to remove a connection.}\label{fig:transitivity_example_remove}
    \end{figure}
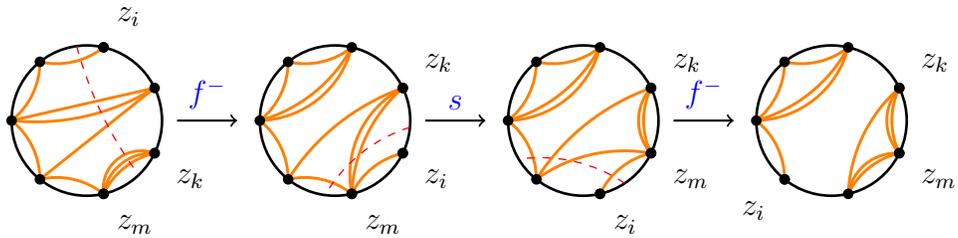

    If the last case is satisfied, apply $f^-_{z_i, z_k}$ and repeat the previous step except for points $y_{z_i - q}$ and applying $f^+_{z_i, z_i - q - 1}$. If we still end with the last case, all connections are either to point $z_k$ or $z_m$. Let us define $p$ to be the current number of connections between $z_k$ and $z_m$. Then there is no arc diagram satisfying the assumptions for the given $p$ since $z_i$ will always be connected to $z_k$ or $z_m$. 
    Furthermore $\sum_{l_j \mid j \neq k, j\neq m} l_j = l_k + l_m - 2p$, so a diagram with smaller $p$ cannot exist either. 
\end{proof}

\begin{lemma}\label{lemma:add_connection}
    Consider an initial arc diagram $x \in X(l_1, l_2,\dots,l_n, l_\infty)$ with $n \geq 3$. Let us select an arbitrary pair of neighboring points $z_k < z_m$, neither of which is connected to $z_i$. 
    
    Assume there exists $y \in X(l_1, l_2,\dots,l_n, l_\infty)$ where the number of connections between $z_k$ and $z_m$ is larger than in $x$ and the points $z_k$ and $z_m$ are not connected to $z_i$. Then there exists $j$ in the cactus group such that $jx$ has one more connection between $z_k$ and $z_m$ and neither point is connected to $z_i$. 
    Furthermore, the order of the points other than $z_i$ in $jx$ is the same as in $x$. 
\end{lemma}

\begin{proof}
    Let the index of the lesser of the two points be $z_k$ and the greater be $z_{m}$. If an arc diagram with another connection between them exists, they must both have a nonzero number of connections going outside of the these two points. 
    
    Let $u$ be some point that is not connected to both $z_k$ and $z_m$. Such a point must exist since if it does not, $n$ must be 3, which contradicts our assumption. Now, we can perform a fissure operation to attach $z_i$ to $u$ (whether it would be an $f^+$ or $f^-$ depends on which side of the non $z_k$ or $z_m$ point to which $u$ is connected $z_i$ lies on). We then consider two cases:
    
    \begin{itemize}
    \item If $u$ is not connected to $z_m$ and the shear line of $s_{z_i, z_m}$ crosses at least one of the arcs leaving $z_m$, we then perform $f^+_{z_i, z_m}$ followed by $s_{z_m, z_i}$. Let $r$ be some point to which $z_k$ is connected that is not $z_m$. We perform $f^+_{z_i, r}$ and we are done.
    \item Otherwise, we perform $f^-_{z_i, z_k - 1}$ followed by $s_{z_i, z_k}$. Let $r$ be some point to which $z_m$ is connected which is not $z_k$. We perform $f^-_{z_i, r - 1}$ and we are done. 
    \end{itemize}

    For an example of the first case, see Figure~\ref{fig:transitivity_example_add}.

    \begin{figure}
        \begin{center}
            \begin{tikzpicture}
                \def\radius{1.4cm}
                \pgfmathtruncatemacro\numPoints{6} 
                \begin{scope}[shift={(0,0)}]

                \draw[line width=1pt] (0,0) circle (\radius);
                
                \pgfmathsetmacro\angle{360/\numPoints}
                
                \foreach \i in {1,...,\the\numexpr\numPoints-1\relax} {
                    \coordinate (P\i) at (-1*\angle*\i+180:\radius);
                }
                \coordinate (P0) at (180:\radius);
                
                
                \draw[orange, line width=1.1pt] (P0) arc (300.0:360.0:1.0*\radius); 
\draw[orange, line width=1.1pt] (P0) arc (85.65890627325528:34.34109372674472:2.0*\radius); 
\draw[orange, line width=1.1pt] (P0) arc (60.0:3.552713678800501e-15:1.0*\radius); 
\draw[orange, line width=1.1pt] (P0) arc (80.28486276817378:-20.284862768173788:0.65*\radius); 
\draw[orange, line width=1.1pt, -] (P1) -- (P4); 
\draw[orange, line width=1.1pt] (P2) arc (180.0:240.0:1.0*\radius); 
\draw[orange, line width=1.1pt] (P3) arc (120.0:180.0:1.0*\radius); 
\draw[orange, line width=1.1pt] (P4) arc (60.0:120.0:1.0*\radius); 
\draw[orange, line width=1.1pt] (P4) arc (39.71513723182621:140.28486276817378:0.65*\radius); 
\draw[orange, line width=1.1pt] (180-4*\angle:\radius) arc (80.40593177313954:99.59406822686046:3*\radius); 

\draw[red, dashed, line width=0.5pt] (180-1.6*\angle:\radius) arc (186.13966928521052:233.86033071478948:2*\radius); 


                
                \fill (P0) circle (2pt) node[label=-1*\angle*0+180:$r$] {};
                \fill (P1) circle (2pt) node[label=-1*\angle*1+180:$u$] {};
                \fill (P2) circle (2pt) node[label=-1*\angle*2+180:$z_i$] {};
                \fill (P3) circle (2pt) node[label=-1*\angle*3+180:] {};
                \fill (P4) circle (2pt) node[label=-1*\angle*4+180:$z_k$] {};
                \fill (P5) circle (2pt) node[label=-1*\angle*5+180:$z_m$] {};

                \end{scope}

                \begin{scope}[shift={(4,0)}]

                    \draw[line width=1pt] (0,0) circle (\radius);
                    
                    \pgfmathsetmacro\angle{360/\numPoints}
                    
                    \foreach \i in {1,...,\the\numexpr\numPoints-1\relax} {
                        \coordinate (P\i) at (-1*\angle*\i+180:\radius);
                    }
                    \coordinate (P0) at (180:\radius);
                    
                    
                    \draw[orange, line width=1.1pt] (P0) arc (300.0:360.0:1.0*\radius); 
\draw[orange, line width=1.1pt] (P0) arc (85.65890627325528:34.34109372674472:2.0*\radius); 
\draw[orange, line width=1.1pt] (P0) arc (60.0:3.552713678800501e-15:1.0*\radius); 
\draw[orange, line width=1.1pt] (P0) arc (80.28486276817378:-20.284862768173788:0.65*\radius); 
\draw[orange, line width=1.1pt, -] (P1) -- (P4); 
\draw[orange, line width=1.1pt] (P2) arc (180.0:240.0:1.0*\radius); 
\draw[orange, line width=1.1pt] (P4) arc (60.0:120.0:1.0*\radius); 
\draw[orange, line width=1.1pt] (P4) arc (39.71513723182621:140.28486276817378:0.65*\radius); 
\draw[orange, line width=1.1pt] (180-4*\angle:\radius) arc (80.40593177313954:99.59406822686046:3*\radius); 
\draw[orange, line width=1.1pt] (P2) arc (154.34109372674473:205.65890627325527:2.0*\radius); 

\draw[red, dashed, line width=0.5pt] (180-0.6*\angle:\radius) arc (246.13966928521052:293.8603307147895:2*\radius); 


                    \fill (P0) circle (2pt) node[label=-1*\angle*0+180:$r$] {};
                    \fill (P1) circle (2pt) node[label=-1*\angle*1+180:$u$] {};
                    \fill (P2) circle (2pt) node[label=-1*\angle*2+180:] {};
                    \fill (P3) circle (2pt) node[label=-1*\angle*3+180:$z_i$] {};
                    \fill (P4) circle (2pt) node[label=-1*\angle*4+180:$z_k$] {};
                    \fill (P5) circle (2pt) node[label=-1*\angle*5+180:$z_m$] {};

                \end{scope}

                \begin{scope}[shift={(8,0)}]

                    \draw[line width=1pt] (0,0) circle (\radius);
                    
                    \pgfmathsetmacro\angle{360/\numPoints}
                    
                    \foreach \i in {1,...,\the\numexpr\numPoints-1\relax} {
                        \coordinate (P\i) at (-1*\angle*\i+180:\radius);
                    }
                    \coordinate (P0) at (180:\radius);
                    
                    
                    \draw[orange, line width=1.1pt] (P0) arc (274.3410937267447:325.6589062732553:2.0*\radius); 
\draw[orange, line width=1.1pt] (P0) arc (85.65890627325528:34.34109372674472:2.0*\radius); 
\draw[orange, line width=1.1pt] (P0) arc (60.0:3.552713678800501e-15:1.0*\radius); 
\draw[orange, line width=1.1pt] (P0) arc (80.28486276817378:-20.284862768173788:0.65*\radius); 
\draw[orange, line width=1.1pt] (P1) arc (240.0:300.0:1.0*\radius); 
\draw[orange, line width=1.1pt] (P3) arc (120.0:180.0:1.0*\radius); 
\draw[orange, line width=1.1pt] (P3) arc (99.71513723182622:200.28486276817378:0.65*\radius); 
\draw[orange, line width=1.1pt] (P4) arc (60.0:120.0:1.0*\radius); 
\draw[orange, line width=1.1pt] (P4) arc (39.71513723182621:140.28486276817378:0.65*\radius); 
\draw[orange, line width=1.1pt] (180-4*\angle:\radius) arc (80.40593177313954:99.59406822686046:3*\radius); 
                    
\draw[red, dashed, line width=0.5pt] (180--0.4*\angle:\radius) arc (306.1396692852105:353.8603307147895:2*\radius); 

                    \fill (P0) circle (2pt) node[label=-1*\angle*0+180:$r$] {};
                    \fill (P1) circle (2pt) node[label=-1*\angle*1+180:$z_i$] {};
                    \fill (P2) circle (2pt) node[label=-1*\angle*2+180:$u$] {};
                    \fill (P3) circle (2pt) node[label=-1*\angle*3+180:] {};
                    \fill (P4) circle (2pt) node[label=-1*\angle*4+180:$z_k$] {};
                    \fill (P5) circle (2pt) node[label=-1*\angle*5+180:$z_m$] {};
        
                \end{scope}

                \begin{scope}[shift={(0,-4.5)}]

                    \draw[line width=1pt] (0,0) circle (\radius);
                    
                    \pgfmathsetmacro\angle{360/\numPoints}
                    
                    \foreach \i in {1,...,\the\numexpr\numPoints-1\relax} {
                        \coordinate (P\i) at (-1*\angle*\i+180:\radius);
                    }
                    \coordinate (P0) at (180:\radius);
                    
                    
                    \draw[orange, line width=1.1pt] (P0) arc (60.0:3.552713678800501e-15:1.0*\radius); 
                    \draw[orange, line width=1.1pt] (P1) arc (240.0:300.0:1.0*\radius); 
                    \draw[orange, line width=1.1pt] (P1) arc (219.71513723182622:320.2848627681738:0.65*\radius); 
                    \draw[orange, line width=1.1pt, -] (P1) -- (P4); 
                    \draw[orange, line width=1.1pt] (P1) arc (25.658906273255283:-25.658906273255283:2.0*\radius); 
                    \draw[orange, line width=1.1pt] (P3) arc (120.0:180.0:1.0*\radius); 
                    \draw[orange, line width=1.1pt] (P3) arc (99.71513723182622:200.28486276817378:0.65*\radius); 
                    \draw[orange, line width=1.1pt] (P4) arc (60.0:120.0:1.0*\radius); 
                    \draw[orange, line width=1.1pt] (P4) arc (39.71513723182621:140.28486276817378:0.65*\radius); 
                    \draw[orange, line width=1.1pt] (180-4*\angle:\radius) arc (80.40593177313954:99.59406822686046:3*\radius); 
                    
                    \draw[red, dashed, line width=0.5pt] (180--1.4*\angle:\radius) arc (366.1396692852105:413.8603307147895:2*\radius); 
                            
                    
                    \fill (P0) circle (2pt) node[label=-1*\angle*0+180:$z_i$] {};
                    \fill (P1) circle (2pt) node[label=-1*\angle*1+180:$r$] {};
                    \fill (P2) circle (2pt) node[label=-1*\angle*2+180:$u$] {};
                    \fill (P3) circle (2pt) node[label=-1*\angle*3+180:] {};
                    \fill (P4) circle (2pt) node[label=-1*\angle*4+180:$z_k$] {};
                    \fill (P5) circle (2pt) node[label=-1*\angle*5+180:$z_m$] {};
        
                \end{scope}

                \begin{scope}[shift={(4,-4.5)}]

                    \draw[line width=1pt] (0,0) circle (\radius);
                    
                    \pgfmathsetmacro\angle{360/\numPoints}
                    
                    \foreach \i in {1,...,\the\numexpr\numPoints-1\relax} {
                        \coordinate (P\i) at (-1*\angle*\i+180:\radius);
                    }
                    \coordinate (P0) at (180:\radius);
                    
                    
                    \draw[orange, line width=1.1pt] (P0) arc (300.0:360.0:1.0*\radius); 
                    \draw[orange, line width=1.1pt] (P0) arc (85.65890627325528:34.34109372674472:2.0*\radius); 
                    \draw[orange, line width=1.1pt] (P0) arc (95.26438968275465:24.735610317245346:1.5*\radius); 
                    \draw[orange, line width=1.1pt] (180-0*\angle:\radius) arc (76.77865488096036:43.22134511903964:3*\radius); 
                    \draw[orange, line width=1.1pt] (P0) arc (60.0:3.552713678800501e-15:1.0*\radius); 
                    \draw[orange, line width=1.1pt] (P1) arc (240.0:300.0:1.0*\radius); 
                    \draw[orange, line width=1.1pt] (P1) arc (219.71513723182622:320.2848627681738:0.65*\radius); 
                    \draw[orange, line width=1.1pt, -] (P1) -- (P4); 
                    \draw[orange, line width=1.1pt] (P3) arc (120.0:180.0:1.0*\radius); 
                    \draw[orange, line width=1.1pt] (P3) arc (99.71513723182622:200.28486276817378:0.65*\radius); 
                    
                    \draw[red, dashed, line width=0.5pt] (180--0.4*\angle:\radius) arc (306.1396692852105:353.8603307147895:2*\radius); 
                            
                    
                    \fill (P0) circle (2pt) node[label=-1*\angle*0+180:$z_m$] {};
                    \fill (P1) circle (2pt) node[label=-1*\angle*1+180:$r$] {};
                    \fill (P2) circle (2pt) node[label=-1*\angle*2+180:$u$] {};
                    \fill (P3) circle (2pt) node[label=-1*\angle*3+180:] {};
                    \fill (P4) circle (2pt) node[label=-1*\angle*4+180:$z_k$] {};
                    \fill (P5) circle (2pt) node[label=-1*\angle*5+180:$z_i$] {};
        
                \end{scope}

                \begin{scope}[shift={(8,-4.5)}]

                    \draw[line width=1pt] (0,0) circle (\radius);
                    
                    \pgfmathsetmacro\angle{360/\numPoints}
                    
                    \foreach \i in {1,...,\the\numexpr\numPoints-1\relax} {
                        \coordinate (P\i) at (-1*\angle*\i+180:\radius);
                    }
                    \coordinate (P0) at (180:\radius);
                    
                    
                    \draw[orange, line width=1.1pt] (P0) arc (300.0:360.0:1.0*\radius); 
\draw[orange, line width=1.1pt] (P0) arc (274.3410937267447:325.6589062732553:2.0*\radius); 
\draw[orange, line width=1.1pt] (P0) arc (258.2275924334342:341.7724075665658:1.3*\radius); 
\draw[orange, line width=1.1pt] (P0) arc (60.0:3.552713678800501e-15:1.0*\radius); 
\draw[orange, line width=1.1pt] (P3) arc (120.0:180.0:1.0*\radius); 
\draw[orange, line width=1.1pt] (P3) arc (99.71513723182622:200.28486276817378:0.65*\radius); 
\draw[orange, line width=1.1pt] (P4) arc (60.0:120.0:1.0*\radius); 
\draw[orange, line width=1.1pt] (P4) arc (39.71513723182621:140.28486276817378:0.65*\radius); 
\draw[orange, line width=1.1pt]  (180-4*\angle:\radius) arc (24.619977328657114:155.38002267134289:0.55*\radius); 
\draw[orange, line width=1.1pt] (180-4*\angle:\radius) arc (80.40593177313954:99.59406822686046:3*\radius); 

                    
                    \fill (P0) circle (2pt) node[label=-1*\angle*0+180:$r$] {};
                    \fill (P1) circle (2pt) node[label=-1*\angle*1+180:$z_i$] {};
                    \fill (P2) circle (2pt) node[label=-1*\angle*2+180:$u$] {};
                    \fill (P3) circle (2pt) node[label=-1*\angle*3+180:] {};
                    \fill (P4) circle (2pt) node[label=-1*\angle*4+180:$z_k$] {};
                    \fill (P5) circle (2pt) node[label=-1*\angle*5+180:$z_m$] {};
        
                \end{scope}

                \draw[thick, ->] (1.6, 1) -- (2.4, 1)  node[midway, above, text=blue] {$s$};

                \draw[thick, ->] (1.6+4, 1) -- (2.4+4, 1)  node[midway, above, text=blue] {$f^-$};

                \draw[thick, ->, looseness=.3] (8, -1.6) to[out=-90, in=90] node[above, text=blue]{$f^+$} (0, -4.5+1.6);

                \draw[thick, ->] (1.6, -3.5) -- (2.4, -3.5)  node[midway, above, text=blue] {$s$};

                \draw[thick, ->] (1.6+4, -3.5) -- (2.4+4, -3.5)  node[midway, above, text=blue] {$f^+$};


            \end{tikzpicture}
        \end{center}
        \caption{A simple example of actions required to add a connection.}\label{fig:transitivity_example_add}
    \end{figure}
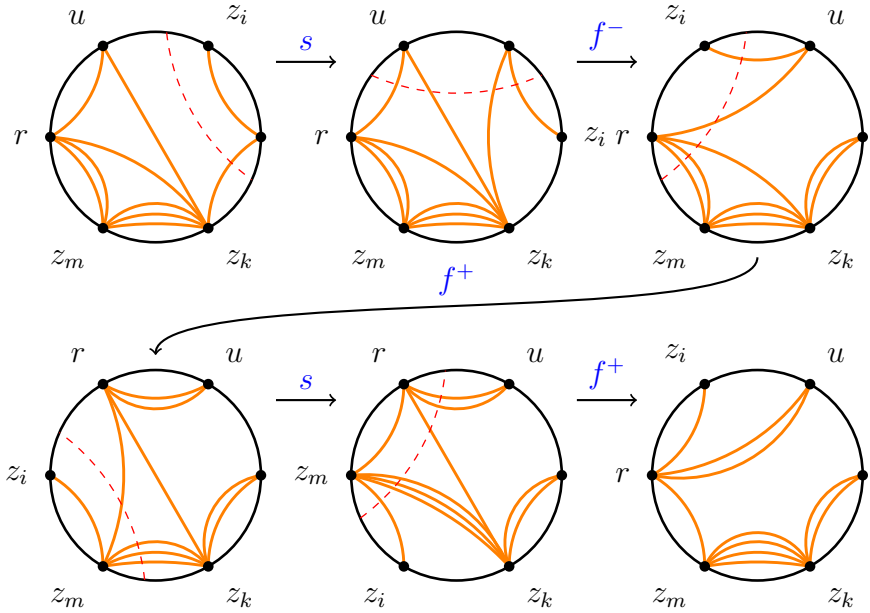
\end{proof}

\begin{theorem}\label{thm:transitivity}
    $J_n$ acts transitively on a set $X(l_1, l_2, \dots, l_n, l_\infty)$ when there exists some $l_i = 1$ (or $l_\infty = 1$).
\end{theorem}

\begin{proof}
    We proceed by induction on the number of points. The base case is when $n = 3$, in which case all elements of $X(l_1, l_2, l_3, l_\infty)$ are essentially the same except up to where the point $z_i$ is and where it is connected (since an arc diagram with 3 points and given valences is unique). Given some configuration, we can change which point $z_i$ is connected to by switching it with the point adjacent to it which it is not connected to. 
    Since we can also change which point $z_i$ is adjacent to but not connected to by switching it with the point to which it is connected, we can attach $z_i$ to whichever point we want. Then, the order can be established by switching it with the point to which it is connected.

    If $n > 3$, we will have some initial configuration and a target configuration. By Proposition~\ref{prop:homomorphism_to_S}, we can act on the initial configuration to get the points to be in the same order as the target configuration. 

    Now, we pick a pair of points neither of which is connected to $z_i$ in either the initial or the final configuration (since $n$ is at least 4, such a pair exists). Either the target configuration has more arcs between this pair, less arcs between this pair or they both have the same number of connections between this pair of points. In either case, we know that such a configuration exists (namely, the target configuration). Thus, by Lemma~\ref{lemma:remove_connection} and Lemma~\ref{lemma:add_connection}, we can achieve the target number of connections between the two points. 

    Then, we can treat these two points as if they are a single point in both the initial and target configuration. Since actions $s_{p, q}$ do not affect connections between points $y_a$ and $y_b$ for $p \leq a < b \leq q$, if we treat these two neighboring points a single point, the number of connections between them will not change when we do $J_n$ actions on them. Thus, we have reduced the number of points and we can proceed by induction. 
\end{proof}

\section{Actions of \texorpdfstring{$J_n$}{Jn} on the Set \texorpdfstring{$X(2, 2, \dots, 2)$}{X(2,2,...,2)}}\label{sec:single_invariant}

Given an arc diagram with all even valences, we can define a property of the diagram that we call the \emph{number of components}. We calculate this number by picking an arbitrary point to start with and then splitting the connecting lines in the diagram into closed, non overlapping loops according to the right-hand rule. We call the number of points in a component its \emph{size}. See Figure~\ref{fig:right_hand_rule_example} for an example.

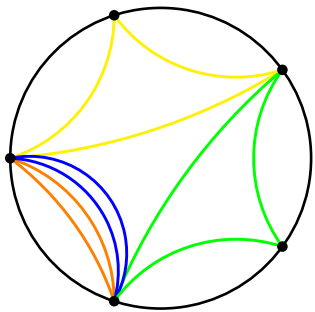
\begin{figure}
    \begin{center}
        \begin{tikzpicture}

            \def\radius{2cm}
            \pgfmathtruncatemacro\numPoints{5} 
            
            \draw[line width=1pt] (0,0) circle (\radius);
            
            \pgfmathsetmacro\angle{360/\numPoints}
            
            \foreach \i in {1,...,\the\numexpr\numPoints-1\relax} {
                \coordinate (P\i) at (-1*\angle*\i+180:\radius);
            }
            \coordinate (P0) at (180:\radius);
            
            
            \draw[yellow, line width=1.1pt] (P0) arc (288.0:360.0:1.0*\radius); 
            \draw[yellow, line width=1.1pt] (180-0*\angle:\radius) arc (274.2453850659197:301.7546149340803:4*\radius); 
            \draw[green, line width=1.1pt] (180-2*\angle:\radius) arc (130.24538506591975:157.75461493408025:4*\radius); 
            \draw[orange, line width=1.1pt] (P0) arc (72.00000000000001:-1.4210854715202004e-14:1.0*\radius); 
            \draw[blue, line width=1.1pt] (P0) arc (87.6018644658794:-15.601864465879409:0.75*\radius); 
            \draw[blue, line width=1.1pt] (P0) arc (101.4026940179763:-29.4026940179763:0.6464466094067263*\radius); 
            \draw[orange, line width=1.1pt] (180-0*\angle:\radius) arc (53.091146270135496:18.908853729864507:2*\radius); 
            \draw[yellow, line width=1.1pt] (P1) arc (216.0:288.0:1.0*\radius); 
            \draw[green, line width=1.1pt] (P2) arc (144.0:216.0:1.0*\radius); 
            
            \draw[green, line width=1.1pt] (P3) arc (72.0:144.0:1.0*\radius); 

            \foreach \i in {1,...,\the\numexpr\numPoints-1\relax} {
                \fill (P\i) circle (2pt);
            }
            \fill (P0) circle (2pt);

            
            \end{tikzpicture}
    \end{center}
    \caption{An example with 4 components.}\label{fig:right_hand_rule_example}
\end{figure}

This leads to an important lemma.

\begin{lemma}\label{lemma:components_invariant}
    The number of components is an invariant under actions of the group $J_n$ when all $l_i$ are even.
\end{lemma}
\begin{proof}
    Consider the intersections between arcs and the shear line of any given transformation given by $s_{i,j} \in J_n$. These intersections can be divided exactly into neighboring pairs of intersections both of which correspond to the same component. When the action is applied, each pair is matched up with another pair from a different component, so each loop that was broken remains a closed loop. Thus, the number of components remains unchanged. 
\end{proof}

\begin{theorem}\label{thm:l_n=2}
    The $J_n$ orbits over the set $X(2, 2,\dots,2)$ are completely characterized by the number of components. That is, there is exactly one orbit for every possible number of components.
\end{theorem}

\begin{proof}
    One direction is given by Lemma~\ref{lemma:components_invariant}. To prove the other direction, we must prove that for any $a, b \in X(2, 2, \dots, 2)$ with the same number of components, we have a $g \in J_n$ such that $a = gb$. 

    Throughout this proof, note that the points all have the same valence so their order does not matter by Corollary~\ref{corr:order_doesn't_matter}. 
    
    Consider a component in diagram $a$ which has only one connecting line for which there are any arcs between it and the edge of the diagram (in a sense, it lies flush against the edge of the diagram). For example, the yellow, blue and purple components in Figure~\ref{fig:right_hand_rule_example} satisfy this condition while the green and red components do not. Let us call the endpoints of the external connecting line $z_j$ and $z_k$ where $z_j < z_k$.

    To transform $b$ into $a$, we first find $z_j$ in $b$ and transform the component it belongs to so that it also lies flush against the edge with $z_j$ as one of its endpoints. To do this, we apply the following algorithm until the condition is satisfied:
    \begin{itemize}
        \item Start with $c = z_j$ and increment $c$ until the point $c$ is no longer part of the same component.
        \item Define $d$ as the next point which is part of the same component as $z_j$.
        \item Apply $s_{c, d}$.
    \end{itemize}

    Now, there are three cases: the component that $z_j$ is the endpoint of in $b$ is larger than in $a$, it is smaller than in $a$, or they are the same size. If they are the same size, this step can be skipped.

    If we need to decrease the size of the component in $b$, first we pick a different component with endpoints $e$ and $g$ with $e < g$. Next, we apply $s_{z_j + q - 1, e}$ where $q$ is the size of the component in $a$.

    To increase the size of the component in $b$, we show that we can always increment it by 1. We locate the smaller endpoint of a component with size greater than 2 in diagram $b$ and denote it $e$. Such a component must exist since if it did not, the size of the component containing $z_j$ would already be maximal and a diagram $a$ with the same number of components but a bigger component containing $z_j$ could not exist. 
    We transform the component with endpoint $e$ in the same way we transformed the component containing $z_j$ so that it lay flush against the edge of the diagram.
    
    We denote the greater endpoint of the component containing $z_j$ in diagram $b$ as $z_m$. Finally, we apply the transformation $s_{z_m, e + 1}$, increasing the size of the component with endpoint $z_j$ by 1. 

    To complete the argument, note that we can treat the whole component that now matches between the two diagrams as a point with valence 0 (which can be moved around the diagram freely by switching it with its neighbor), and proceed inductively until we are left with a single component (which will have to be the same between the two diagrams).
\end{proof}

An interesting consequence of this is that we can calculate the number of orbits for any $n$ where all $l_i = 2$.

\begin{corollary}\label{corr:counting_orbits}
    The number of orbits given by actions of $J_n$ on $X(2,2,2,\dots,2)$ is $\left \lfloor{n/2}\right \rfloor$.
\end{corollary}

\begin{proof}
    There can always be 1 component but there cannot be more than $n/2$ components. Every value in between is possible since two components can be combined to reduce the number of components. 
\end{proof}

\section{Additional Relation on the Action of \texorpdfstring{$J_n$}{Jn}}\label{sec:simple_relation}

The next result follows from a general proof in \cite{chmutov2017berensteinkirillov}. We present a novel proof of this specifically for arc diagrams. 

\begin{theorem}\label{thm:braid_relation}
    When $J_n$ acts on the set $X(l_1, l_2, \dots, l_n, l_\infty)$, the braid relation, \[s_{i, i+1} s_{i-1, i} s_{i, i+1} = s_{i-1, i} s_{i, i+1} s_{i-1, i},\] is always satisfied. 
\end{theorem}

\begin{proof}
    We use casework to prove this result.

    Since the actions we are interested in only affect points at indices $i - 1$ to $i + 1$, we can represent the actions we perform by acting on an arc diagram with $n = 3$. In particular, $z_1, z_2$ and $z_3$ would be the points at indices $i-1$, $i$, and $i+1$, respectively, and $z_\infty$ would represent all other points in the original diagram.

    With such a setup, there are two distinct starting diagrams, shown on the left of Figure~\ref{fig:s_proof_diagram}. Now, we can calculate all possible end results after applying $s_{i, i+1} s_{i-1, i} s_{i, i+1}$, as shown in Figure~\ref{fig:s_proof_diagram}.

    \begin{figure}
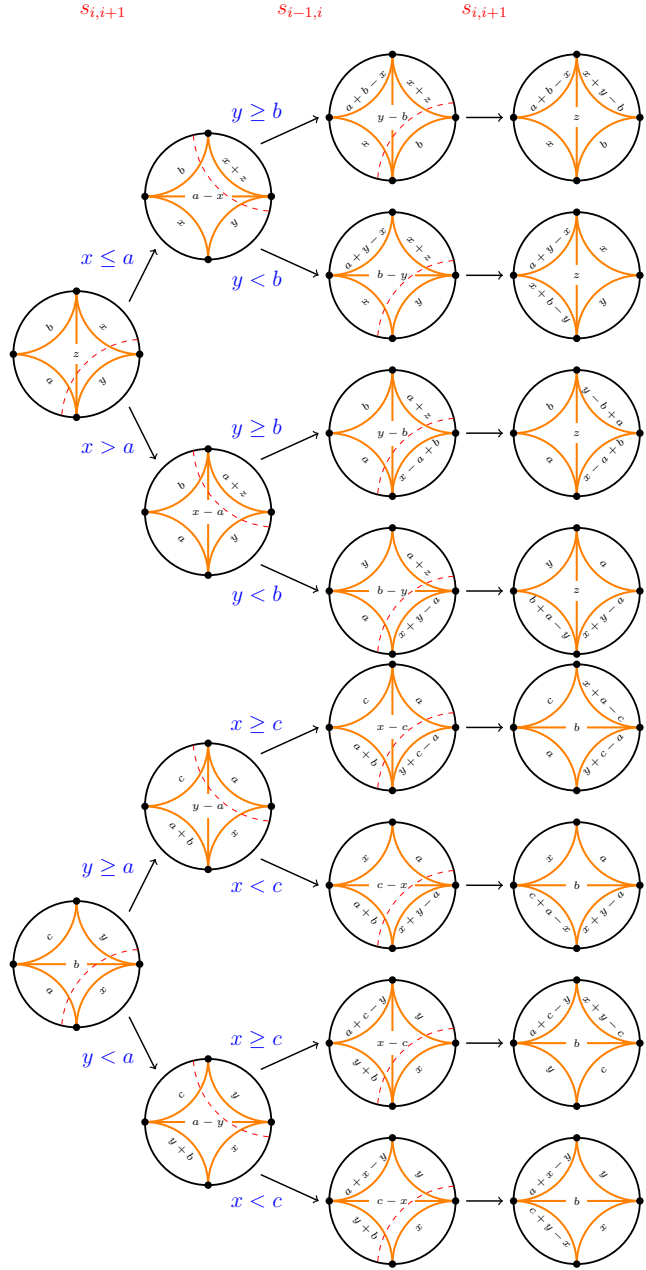

        \begin{center}
            \scalebox{0.7}{

            }
        \end{center}
        \caption{A representations of all possible results of applying $s_{i, i+1} s_{i-1, i} s_{i, i+1}$. The labels on each connection represent the number of connecting lines passing between that pair of points.}\label{fig:s_proof_diagram}
    \end{figure}
    
    Now, notice that $s_{i-1, i+1} s_{i, i+1} s_{i-1, i} s_{i, i+1} s_{i-1, i+1} = s_{i-1, i} s_{i, i+1} s_{i-1, i}$ by the defining relations of $J_n$. Thus, to find the result of applying $s_{i-1, i} s_{i, i+1} s_{i-1, i}$, we can use the same computations used in Figure~\ref{fig:s_proof_diagram}, except we must reflect the diagram over the horizontal line passing through its center before applying the first action and after applying the last action. 

    In each of the 8 cases in Figure~\ref{fig:s_proof_diagram}, we can manually verify that the effect of applying $s_{i, i+1} s_{i-1, i} s_{i, i+1}$ and $s_{i-1, i} s_{i, i+1} s_{i-1, i}$ is the same for any starting conditions, so the relation $s_{i, i+1} s_{i-1, i} s_{i, i+1} = s_{i-1, i} s_{i, i+1} s_{i-1, i}$ holds. 
\end{proof}

In the case when $l_1 = l_2 = \dots = l_\infty$ there is another relation that can be described.

\begin{theorem}\label{thm:second_relation}
    When $J_n$ acts on the set $X(l_1, l_2, \dots, l_n, l_\infty)$ where $l_1 = l_2 = \dots = l_\infty$, the relation \[(s_{1,n}s_{1,n-1})^{n(n+1)} = e,\] where $e$ is the identity element, is always satisfied. 
\end{theorem}

\begin{proof}
    We consider the effect of this action on the order of the points and the connecting lines within the arc diagram separately.

    The order of the points is periodic every $n$ applications of $s_{1,n}s_{1,n-1}$ since this transformation rotates all of the points except for $z_\infty$ by one position clockwise.

    Now, we consider the effect of $s_{1, n-1}$ on the connecting lines. We claim that this transformation is equivalent to reflecting the whole diagram over the diameter which connects the gap between points $z_n$ and $z_\infty$ with the opposite side of the diagram (either $z_{n/2}$ or the region between $z_{(n+1)/2}$ and $z_{(n-1)/2}$, depending on the parity of $n$).
    
    In the case of the active region, this is true because that is how the action of $J_n$ is defined. The two points outside of the active region ($z_n$ and $z_\infty$) both have the same number of connecting lines intersecting the shear line, so they are identical from the point of view of the active region. Thus, while they do not get explicitly reflected by $s_{1, n-1}$, the reflection would have not effect on them regardless so $s_{1,n-1}$ is equivalent to the reflection described.

    Next, we claim that the effect of $s_{1, n}$ is equivalent to reflecting over the diameter connecting $z_\infty$ and the opposite side of the circle (either $z_{(n+1)/2}$ or the region between $z_{n/2}$ and $z_{(n+2)/2}$, depending on the parity of $n$). This follows directly from the definition of the transformation corresponding to $s_{1, n}$. 

    In summary, the transformation corresponding to $s_{1, n} s_{1, n-1}$ is equivalent to two reflections over diameters of the diagram with an angle of $\frac{\pi}{n+1}$ between them. This is also the same as a rotation of $\frac{2\pi}{n+1}$. Thus, repeating this action $n+1$ times is the identity transformation. 

    Combining these two aspects of the diagram and their respective periodicities, we find that $(s_{1,n}s_{1,n-1})^{n(n+1)} = e$.
\end{proof}

\printbibliography

\end{document}